\documentclass[12pt,a4paper]{article}

\usepackage{geometry}
 \makeatletter
\def\@cline#1-#2\@nil{%
  \omit
  \@multicnt#1%
  \advance\@multispan\m@ne
  \ifnum\@multicnt=\@ne\@firstofone{&\omit}\fi
  \@multicnt#2%
  \advance\@multicnt-#1%
  \advance\@multispan\@ne
  \leaders\hrule\@height\arrayrulewidth\hfill
  \cr
  \noalign{\nobreak\vskip-\arrayrulewidth}}
\makeatother
\usepackage{graphicx}
\usepackage{caption}
\usepackage{subcaption}
\usepackage{xcolor}
\usepackage{amssymb}

\usepackage{comment}
\usepackage[utf8]{inputenc}

\usepackage{setspace}
\usepackage{calligra,amsmath}

\usepackage[hyperfootnotes=false]{hyperref}
\hypersetup{
    colorlinks=true,
    linkcolor=blue,
    filecolor=magenta,      
    urlcolor=cyan,
}

\usepackage{setspace}
\usepackage{amsthm}
\usepackage{amsfonts}
\usepackage{mathtools}
\usepackage{bm}

\usepackage[]{algpseudocode}
\usepackage{algorithm}

\makeatletter

\usepackage{longtable}
    \usepackage{multirow}
\usepackage{siunitx,booktabs}
\usepackage{authblk}
\usepackage{tablefootnote}
\usepackage[toc,page]{appendix}

\numberwithin{Theorem}{section}
\numberwithin{Definition}{section}
\numberwithin{Lemma}{section}
\numberwithin{Algorithm}{section}
\numberwithin{equation}{section}
\numberwithin{Proposition}{section}
\newtheorem{assumption}{Assumption}
\newtheorem{remark}{Remark}[section]

\newtheorem{theorem}{Theorem}[section]

\newtheorem{lemma}[theorem]{Lemma} 

\newtheorem{definition}[theorem]{Definition}
\newtheorem{proposition}[theorem]{Proposition}

 \usepackage{caption}
\usepackage{mathdots}
\usepackage[font=sf, labelfont={sf,bf}, margin=1cm]{caption}

\usepackage {mathrsfs}
\usepackage{tensor}
\def\DC#1#2#3{\tensor*[^{\mathrm{C}}_{#2}]{\mathop{D}\nolimits}{^{#1}_{#3}}}
\def\DRL#1#2#3{\tensor*[^{\mathrm{RL}}_{#2}]{\mathop{D}\nolimits}{^{#1}_{#3}}}
\def\DR#1#2{\tensor*[^{\mathrm{R}}]{\mathop{D}\nolimits}{^{#1}_{#2}}}

\begin{document}
	\title{Fast Solution Methods for Convex Fractional Differential Equation Optimization}
\author[*]{Spyridon Pougkakiotis}
\author[*]{John W. Pearson}
\author[*]{Santolo Leveque}
\author[*]{Jacek Gondzio}
\affil[*]{School of Mathematics, University of Edinburgh\newline \newline Technical Report ERGO 19--014}
\maketitle

\begin{abstract}
\noindent In this paper, we present numerical methods suitable for solving convex quadratic Fractional Differential Equation (FDE) constrained optimization problems, with box constraints on the state and/or control variables. We develop an Alternating Direction Method of Multipliers (ADMM) framework, which uses preconditioned Krylov subspace solvers for the resulting sub-problems. The latter allows us to tackle a range of Partial Differential Equation (PDE) optimization problems with box constraints, posed on space-time domains, that were previously out of the reach of state-of-the-art preconditioners. In particular, by making use of the powerful Generalized Locally Toeplitz (GLT) sequences theory, we show that any existing GLT structure present in the problem matrices is preserved by ADMM, and we propose some preconditioning methodologies that could be used within the solver, to demonstrate the generality of the approach. Focussing on convex quadratic programs with time-dependent 2-dimensional FDE constraints, we derive multilevel circulant preconditioners, which may be embedded within Krylov subspace methods, for solving the ADMM sub-problems. Discretized versions of FDEs involve large dense linear systems. In order to overcome this difficulty, we design a recursive linear algebra, which is based on the Fast Fourier Transform (FFT). We manage to keep the storage requirements linear, with respect to the grid size $N$, while ensuring an order $N \log N$ computational complexity per iteration of the Krylov solver. We implement the proposed method, and demonstrate its scalability, generality, and efficiency, through a series of experiments over different setups of the FDE optimization problem. 
\end{abstract}
\section{Introduction}
\par Optimization problems with Differential Equations (Partial (PDEs) or Ordinary (ODEs)) as constraints have received a great deal of attention within the applied mathematics and engineering communities, due in particular to their wide applicability across many fields of science. In addition to classical differential equation constraints, one may also use Fractional Differential Equations (FDEs) in order to model processes that could not otherwise be modeled using integer derivatives. In fact, there is a wide and increasing use of FDEs in the literature. Among other processes, FDEs have been used to model viscoelasticity (e.g. \cite{paper_31}), chaotic systems (e.g. \cite{paper_32}), turbulent flow, or anomalous diffusion (e.g. \cite{paper_33}). In particular, since the fractional operator is \emph{non-local}, problems with non-local properties can frequently be modeled accurately using FDEs (see \cite{Podlubny} for an extended review). 
\par Availability of closed form solutions for FDEs is rare, and hence various numerical schemes for solving them have been developed and analyzed in the literature (see \cite{MST,MT1,MT2} for finite difference, and \cite{paper_35,paper_34} for finite element methods). Such numerical schemes produce dense matrices, making the solution or even the storage of FDE-constrained optimization problems extremely difficult for fine grids. Naturally, this behavior is even more severe in the case of multidimensional FDEs. In light of the previous, employing standard (black-box) direct approaches for solving such problems requires $O(N^3)$ operations and $O(N^2)$ storage, where $N$ is the number of grid points. Iterative methods with general purpose preconditioners also suffer from similar issues. 
\par Various specialized solution methods have been proposed in the literature, aiming at lowering the computational and storage cost of solving such problems (see for instance \cite{CD,paper_2,paper_61,HKKS,paper_60,ZZ}). One popular and effective approach is to employ tensor product solvers. Such specialized methods have been proposed for solving high-dimensional FDE-constrained inverse problems with great success, even for very fine discretizations (see for example \cite{paper_2,HKKS} and the references therein). While these solvers are highly scalable (with respect to the grid size), to date they have been tailored solely to problems with specific cost functionals and without additional algebraic constraints. Another popular approach is based on the observation that multidimensional FDEs possess a \emph{multilevel Toeplitz}-like structure. It is well known that such matrices can be very well approximated by banded multilevel Toeplitz (see for example \cite{paper_58,paper_60}) or \emph{multilevel circulant} matrices (see \cite{paper_3,paper_1,paper_57,paper_8,paper_7}). The former are usually sparse and can be inverted using specially designed multigrid or factorization methods, while the latter can be inverted or applied to a vector in only $O(N \log N)$ operations using the Fast Fourier Transform (FFT) (e.g. \cite{paper_16}). The idea is to apply a Krylov subspace solver, supported by a banded Toeplitz or circulant preconditioner, in order to solve the optimality conditions of the problem. One is able to redesign the underlying linear algebra, in order to achieve an $O(N \log N)$ iteration complexity for the Krylov solver, with $O(N)$ overall storage requirements (see for example \cite{paper_8,paper_7,paper_6}). While such solution methods are certainly more general (although usually slower), when compared to tensor product solvers, they remain rather sensitive in terms of the underlying structure. In particular, to the authors' knowledge, no such method has been proposed for the solution of more general FDE optimization problems, for instance those which include box constraints on the state and control variables. We highlight that a time-independent problem, with box constraints on the control, is studied in \cite{DC}, and the authors attempt to solve it using a Limited-memory Broyden--Fletcher--Goldfarb--Shanno (L--BFGS) method.
\par In this paper, we present an optimization method suitable for solving convex quadratic PDE-constrained optimization problems with box constraints on the state and control variables. In particular, we assume that we are given an arbitrary PDE constrained inverse problem, an associated discretization method,
and that the resulting sequences of discrete matrices belong to the class of
Generalized Locally Toeplitz (GLT) sequences (we refer the reader to \cite{GLT_theory_VOL_1,GLT_theory_VOL_2,Tilli_LT} for a comprehensive overview of the powerful GLT theory). Then, we propose the use of an Alternating Direction Method of Multipliers (ADMM) for solving the discretized optimization problems. We employ ADMM in order to separate the equality from the inequality constraints. As a consequence of this choice, we show that the linear systems required to be solved during the iterations of ADMM preserve the GLT structure of the initial problem matrices. Using this structure, we present and analyze some general methodologies for efficiently preconditioning such linear systems, and solving them using an appropriate Krylov subspace method. The Krylov subspace method is in turn, under certain mild assumptions, expected to converge in a number of iterations independent of the grid size. Subsequently, we focus on a certain class of convex quadratic optimization problems with FDE constraints. In particular, we consider time-dependent 2-dimensional FDEs, and we precondition the associated discretized matrices using multilevel circulant preconditioners. We manage to keep the storage requirements linear, with respect to the grid size $N$, while ensuring an order $N\log N$ computational complexity per iteration of the Krylov solver inside ADMM. We implement the proposed method, and demonstrate its robustness and efficiency, through a series of experiments over different setups of the FDE optimization problem.
\par This paper is structured as follows. In Section \ref{section Theoretical background}, we provide the relevant theoretical background as well as the notation used throughout the paper. Subsequently, in Section \ref{Section A structure preserving method} we present the proposed ADMM framework, as well as possible preconditioning strategies that could be used to accelerate the solution of the ADMM sub-problems, given the assumption that the associated matrices possess a GLT structure. In Section \ref{section FDE problem} we present the FDE-constrained optimization problem under consideration. Then, in Section \ref{Section: Toeplitz and Circulant Matrices}, we propose the use of a multilevel circulant preconditioner for approximating multilevel Toeplitz matrices arising from the discretization of the FDE under consideration, while demonstrating that such a preconditioner is effective for the problem at hand. In Section \ref{Implementation and Numerical Results}, we discuss the implementation details of the proposed approach and present some numerical results. Finally, in Section \ref{Section conclusions}, we state our conclusions.
\section{Notation and Theoretical Background} \label{section Theoretical background}

\par In this section, we introduce some notation and provide the theoretical background that will be used in the rest of this manuscript. Firstly, we introduce the notion of $d$-indices which will allow us to compactly represent multilevel matrices. For brevity of presentation, we only discuss the crucial notions that will be used in this paper. A more complete presentation of the notation and theory of this section can be found in \cite{GLT_theory_VOL_1,GLT_theory_VOL_2}. A reader familiar with the theory of GLT sequences can skip directly to Section \ref{Section A structure preserving method}.

\begin{definition}
A multi-index $\bm{i}$ of size $d$ ($d$-index) is a row vector in $\mathbb{Z}^d$ with components $i_1,\ldots,i_d$. Using this notation, we define the following notions:
\begin{itemize}
\item $\bm{0,1,2,}\ldots$, are the row vectors of all zeros, ones, twos, etc.
\item $N(\bm{i}) = \prod_{j=1}^d i_j$ and we write $\bm{i} \rightarrow \infty$ to indicate that $\min(\bm{i}) \rightarrow \infty$.
\item Given two d-indices $\bm{h,k}$, we write $\bm{h} \leq \bm{k}$ to express that $h_j \leq k_j$, for all $j \in \{1,\ldots,d\}$. The d-index range $\bm{h},\ldots,\bm{k}$ is a set of cardinality $N(\bm{k-h+1})$ given by $\{\bm{j} \in \mathbb{Z}^d: \bm{h} \leq \bm{j} \leq \bm{k} \}$. The latter set is assumed to be ordered under the lexicographical ordering, that is:
$$ \bigg[ \ldots \big[ [(j_1,\ldots,j_d)]_{j_d = h_d,\ldots,k_d}\big]_{j_{d-1} = h_{d-1},\ldots,k_{d-1}} \ldots \bigg]_{j_1 = h_1,\ldots,k_1}.$$
\item Let a d-index $\bm{m} \in \mathbb{N}^d$, and define $\bm{x} = [x_{\bm{i}}]_{\bm{i} = \bm{1}}^{\bm{m}}$ ($\bm{X} = [x_{\bm{i},\bm{j}}]_{\bm{i},\bm{j} = 1}^{\bm{m}} $, respectively). Then $\bm{x}$ ($\bm{X}$, respectively) is a vector of size $N(\bm{m})$ (a matrix of size $N(\bm{m}) \times N(\bm{m})$, respectively).
\item Any operation involving d-indices that has no meaning in the vector space $\mathbb{Z}^d$ will be interpreted in a componentwise sense.
\end{itemize}
\end{definition}
\noindent  A matrix $A$ of size $N$ is a $d$-level matrix with level orders $n_1,\ldots,n_d$ if $N = n_1  n_2 \cdots n_d$ and it is partitioned into $n_1^2$ square blocks of size $\frac{N}{n_1}$, each of which is partitioned into $n_2^2$ blocks of size $\frac{N}{n_1 n_2}$, and so on until the last $n_d^2$ blocks of size $1$. Then, $A$ can be written as $A = [A_{\bm{ij}}]_{\bm{i},\bm{j} = \bm{1}}^{\bm{n}}$, where $A_{\bm{ij}} = A_{i_1 j_1;\ldots;i_d j_d}$, for $\bm{i},\bm{j} = \bm{1},\ldots,\bm{n}$.
\par Next, we define the notion of a matrix-sequence, which is a fundamental element for studying the asymptotic spectral behavior of structured matrices arising from some discretization of a physical process. In the rest of this manuscript, given an arbitrary matrix $A$, $\sigma(A)$ denotes the set of singular values of the matrix, while $\lambda(A)$ denotes the set of eigenvalues of the matrix $A$ (given that they exist).

\begin{definition}
A matrix-sequence is a sequence of the form $\{A_n\}_n$, where $n$ varies over some infinite subset of $\mathbb{N}$, $A_n$ is a square matrix of size $d_n$, and $d_n \rightarrow \infty$ as $n \rightarrow \infty$. In particular, a $d$-level matrix sequence is a sequence of the form $\{A_{\bm{n}}\}_n$, where $A_{\bm{n}}$ is a matrix of size $N(\bm{n})\times N(\bm{n})$, $n$ varies over some infinite subset of $\mathbb{N}$, and $\bm{n} = \bm{n}(n) \in \mathbb{N}^d$ is such that $\bm{n} \rightarrow \infty$, as $n \rightarrow \infty$. Given a d-level matrix-sequence $\{A_{\bm{n}}\}_n$, we say that it is sparsely unbounded (and denote that as s.u.) if:
$$\lim_{M \rightarrow \infty} \limsup_{n \rightarrow \infty} \frac{\#\{i \in \{1,\ldots,N(\bm{n})\}: \sigma_{i}(A_{\bm{n}}) > M \}}{N(\bm{n})} = 0,$$
\noindent where $\#S$ denotes the cardinality of a set $S$. Similarly, we say that $\{A_{\bm{n}}\}_n$ is sparsely vanishing (and denote that as s.v.) if:

$$ \lim_{M \rightarrow \infty} \limsup_{n \rightarrow \infty} \frac{\#\{i \in \{1,\ldots,N(\bm{n})\}: \sigma_{i}(A_{\bm{n}}) < 1/M \}}{N(\bm{n})} = 0,$$
\noindent where we assume that $1/\infty = 0$.
\end{definition}

\par An important notion is that of clustering. In order to define it we let, for every $z \in \mathbb{C}$ and any $\epsilon > 0$, $D(z,\epsilon)$ represent the disk with center $z$ and radius $\epsilon$. If $S \subseteq \mathbb{C}$ and $\epsilon > 0$, $D(S,\epsilon)$ denotes the $\epsilon$-expansion of $S$, defined as $D(S,\epsilon) = \bigcup_{z \in S}D(z,\epsilon)$.
\begin{definition} \label{definition: clustering}
Let $\{A_n\}_n$ be a sequence of matrices, with $A_n$ of size $d_n \times d_n$, and let $S \subseteq \mathbb{C}$ be a non-empty subset of $\mathbb{C}$. We say that $\{A_n\}_n$ is strongly clustered at $S$ (in the sense of eigenvalues) if $\forall\ \epsilon >0$ we have:
\[\#\{j \in \{1,\ldots,d_n\} : \lambda_j(A_n) \notin D(S,\epsilon)\} = O(1), \]
\noindent and weakly clustered at $S$ if $\forall\ \epsilon > 0$,
\[ \#\{j \in \{1,\ldots,d_n\} : \lambda_j(A_n) \notin D(S,\epsilon)\} = o(d_n).\]
Clustering in the sense of singular values is defined analogously.
\end{definition}

\par Let $f_m,\ f: D \subseteq \mathbb{R}^d \mapsto \mathbb{C}$ be measurable functions, with respect to the Lebesgue measure $\mu_d$ in $\mathbb{R}^d$. We say that $f_m \rightarrow f$ in measure if, for every $\epsilon > 0$, $\underset{m \rightarrow \infty}{\lim} \mu_d \big(\{|f_m -f| > \epsilon\}\big) = 0.$ Furthermore, $f_m \rightarrow f$ a.e. (almost everywhere) if $\mu_d \big(\{f_m \nrightarrow f\}\big) = 0$.

\begin{lemma} \label{lemma convergence in measure}
Let $f_m,\ g_m,\ f,\ g: D \subseteq \mathbb{R}^d \mapsto \mathbb{C}$ be measurable functions.
\begin{enumerate}
\item If $f_m \rightarrow f$ in measure, then $|f_m| \rightarrow |f|$ in measure.
\item If $f_m \rightarrow f$ in measure and $g_m \rightarrow g$ in measure, then $\alpha f_m + \beta g_m \rightarrow \alpha f + \beta g$ in measure for all $\alpha,\ \beta \in \mathbb{C}$.
\item If $f_m \rightarrow f$ in measure, $g_m \rightarrow g$ in measure, and $\mu_d(D) < \infty$, then $f_m g_m \rightarrow fg$ in measure.
\end{enumerate}
\begin{proof}
This is stated in \cite[Lemma 2.3]{GLT_theory_VOL_1} and proved in \cite[Corollary 2.2.6]{BogachevMeasureTheory}. 
\end{proof}
\end{lemma}
\noindent Let $C_c(\mathbb{C})$ ($C_c(\mathbb{R})$, respectively) be the space of complex-(real-)valued continuous functions defined on $\mathbb{C}$ (or $\mathbb{R}$) with compact support. Given a field $\mathbb{K}$ ($=\mathbb{C}$ or $\mathbb{R}$) and a measurable function $g:D \subset \mathbb{R}^d \mapsto \mathbb{K}$, with $0 <\mu_d(D) < \infty$, define the functional:
$$\phi_g: C_c(\mathbb{K}) \mapsto \mathbb{C},\ \ \phi_g(F) = \frac{1}{\mu_d(D)}\int_D F(g(\bm{x}))\ \mathrm{d}\bm{x}.$$
\begin{definition} \label{Definition Spectral distribution}
Let $\{A_n\}_n$ be a matrix-sequence, with $A_n$ of size $d_n \times d_n$. We say that $\{A_n\}_n$ has an asymptotic eigenvalue (spectral) distribution described by a functional $\phi: C_c(\mathbb{C}) \mapsto \mathbb{C}$, and we write $\{A_n\}_n \sim_{\lambda} \phi$, if:
\[\lim_{n \rightarrow \infty} \frac{1}{d_n} \sum_{j=1}^{d_n} F(\lambda_j(A_n)) = \phi(F),\ \ \forall\ F \in C_c(\mathbb{C}).\]
\noindent If $\phi = \phi_{f}$ for some measurable function $f:D\subset \mathbb{R}^d \mapsto \mathbb{C}$, where $0 <\mu_d(D) < \infty$, we say that $\{A_n\}_n$ has an asymptotic  spectral distribution described by $f$ and we write $\{A_n\}_n \sim_{\lambda} f$. Then, $f$ is referred to as the  eigenvalue (spectral) symbol of $\{A_n\}_n$.
\end{definition}
\noindent We can define the asymptotic singular value distribution of a matrix sequence similarly to Definition \ref{Definition Spectral distribution} (see \cite[Definition 2.1]{paper_30}). In that case, we write $\{A_n\}_n \sim_{\sigma} f$.
\par Below we define two important classes of matrix sequences, namely diagonal sampling and Toeplitz matrix-sequences.

\begin{definition}
Let a function $v: [0,1]^d \mapsto \mathbb{C}$ be given. The $\bm{n}$-th diagonal sampling matrix generated by $v$ is denoted by $D_{\bm{n}}(v)$ and is defined by the following $N(\bm{n})\times N(\bm{n})$ diagonal matrix:
\[  D_{\bm{n}}(v) =\underset{\bm{i} = \bm{1},\ldots,\bm{n}}{\textnormal{diag}} v\bigg(\frac{\bm{i}}{\bm{n}}\bigg).\]
\end{definition}
\begin{definition} \label{Definition Toeplitz matrix}
Given a $d$-index $\bm{n} \in \mathbb{N}^d$, a matrix of the form $[a_{\bm{i}-\bm{j}}]_{\bm{i},\bm{j} = \bm{1}}^{\bm{n}} \in \mathbb{C}^{N(\bm{n})\times N(\bm{n})}$ is called a d-level Toeplitz matrix. Unilevel Toeplitz matrices ($d = 1$) are also defined as matrices that are constant along all of their diagonals.
\end{definition}
\noindent A characterization of Toeplitz matrix-sequences is given by the following Theorem, the proof of which can be found in \cite[Sections 3.1, 3.5]{GLT_theory_VOL_2}. Before that, let us define a useful matrix. Given an arbitrary $n \in \mathbb{N}$ and $k \in \mathbb{Z}$, define the $n\times n$ matrix $J_n^{(k)}$ such that $[J_n^{(k)}]_{ij} = 1$ if $i-j = k$ and $[J_n^{(k)}]_{ij}=0$ otherwise. Given two $d$-indices $\bm{n} \in \mathbb{N}^d$ and $\bm{k} \in \mathbb{Z}^d$, we define $ J_{\bm{n}}^{(\bm{k})} = J_{n_1}^{(k_1)} \otimes J_{n_2}^{(k_2)}\otimes \cdots \otimes J_{n_d}^{(k_d)},$ where $\otimes$ denotes the Kronecker product between two matrices.
\begin{theorem} \label{theorem Toeplitz matrices}
Let a function $f:[-\pi,\pi]^d \mapsto \mathbb{C}$ belonging to $L^1([-\pi,\pi]^d)$ be given, with Fourier coefficients denoted by:
$$f_{\bm{k}} = \frac{1}{(2\pi)^d} \int_{[-\pi,\pi]^d} f(\bm{\theta}) e^{- i \langle\bm{k}, \bm{\theta}\rangle}\ \mathrm{d}\bm{\theta},\ \ \bm{k} \in \mathbb{Z}^d,$$
\noindent where $\langle \bm{k}, \bm{\theta} \rangle = \sum_{i = 1}^d k_i \theta_i$. 
The $\bm{n}$-th ($d$-level) Toeplitz matrix associated with $f$ is defined as:
$$T_{\bm{n}}(f) = [f_{\bm{i}-\bm{j}}]_{\bm{i},\bm{j} = \bm{1}}^{\bm{n}} = \sum_{\bm{k} = -(\bm{n}-\bm{1})}^{\bm{n}-\bm{1}} f_{\bm{k}} J_{\bm{n}}^{(\bm{k})}.$$
\noindent Every $d$-level matrix sequence of the form $\{T_{\bm{n}}(f)\}_{n}$, with $\{\bm{n} = \bm{n}(n)\}_n \subseteq \mathbb{N}^d$ such that $\bm{n} \rightarrow \infty$ as $n \rightarrow \infty$, is called a (d-level) Toeplitz sequence generated by $f$, which in turn is referred to as the generating function of $\{T_{\bm{n}}(f)\}_{n}$. Furthermore, $\{T_{\bm{n}}(f)\}_{n} \sim_{\sigma} f.$ If moreover $f$ is real, then $\{T_{\bm{n}}(f)\}_n \sim_{\lambda} f.$
\end{theorem}
\par A special type of Toeplitz matrices are the circulant matrices, as defined below.
\begin{definition}
A matrix of the form $\big[a_{(\bm{i}-\bm{j})\bmod{ \bm{n}}}\big]_{\bm{i},\bm{j} = \bm{1}}^{\bm{n}} \in \mathbb{C}^{N(\bm{n})\times N(\bm{n})}$, for some $d$-index $\bm{n} \in \mathbb{N}^d$, is called a multilevel ($d$-level) circulant matrix.
\end{definition}
\noindent Given an arbitrary $n \in \mathbb{N}$, define the $n \times n$ matrix $C_n$ such that $[C_n]_{ij}= 1$ if $(i-j) \bmod{n} = 1$ and $[C_n]_{ij}=0$ otherwise. Then, for $\bm{n} \in \mathbb{N}^d$ and $\bm{k} \in \mathbb{Z}^d$, let $C_{\bm{n}}^{\bm{k}} = C_{n_1}^{k_1} \otimes C_{n_2}^{k_2} \otimes \ldots \otimes C_{n_d}^{k_d}$, where $C_{n_i}^{k_i}$ is the previously defined matrix $C_{n}$ raised to the power $k_i$. Let $F_n$ denote the unitary discrete Fourier transform of order $n$. For any $\bm{n} \in \mathbb{N}^d$ let $F_{\bm{n}} = F_{n_1}\otimes \ldots \otimes F_{n_d}$. Below we provide a Theorem characterizing multilevel circulant matrices; its proof can be found in \cite[Section 3.4]{GLT_theory_VOL_2}.
\begin{theorem} \label{theorem Circulant matrices}
The $d$-level circulant matrix admits the following expression:
\[\big[a_{(\bm{i}-\bm{j})\bmod{ \bm{n}}}\big]_{\bm{i},\bm{j} = \bm{1}}^{\bm{n}} = \sum_{\bm{k} = \bm{0}}^{\bm{n}- \bm{1}} a_{\bm{k}} C_{\bm{n}}^{\bm{k}},\]
\noindent where $C_{\bm{n}}^{\bm{k}}$ is as defined earlier. Furthermore, letting any $\bm{r} \in \mathbb{N}^d$ and $c_{-\bm{r}},\ldots,c_{\bm{r}} \in \mathbb{C}$, we have that any linear combination of the form $\sum_{\bm{d} = -\bm{r}}^{\bm{r}}c_{\bm{k}}C_{\bm{n}}^{\bm{k}}$ is a $d$-level circulant matrix. Then,
\[\sum_{\bm{k} = -\bm{r}}^{\bm{r}} c_{\bm{k}}C_{\bm{n}}^{\bm{k}} = F^*_{\bm{n}}\left(\underset{\bm{j} = \bm{0},\ldots,\bm{n}-\bm{1}}{\textnormal{diag}}c\left(\frac{2\pi \bm{j}}{\bm{n}}\right) \right)F_{\bm{n}},\]
\noindent where $c(\bm{\theta}) = \sum_{\bm{k} = -\bm{r}}^{\bm{r}} c_{\bm{k}}e^{i\langle\bm{k},\bm{\theta}\rangle}$, and $F_{\bm{n}}$ is the multilevel discrete Fourier transform. Moreover, $\sum_{\bm{k} = -\bm{r}}^{\bm{r}} c_{\bm{k}}C_{\bm{n}}^{\bm{k}} $ is a normal matrix the spectrum of which is given by:
\[\lambda\bigg( \sum_{\bm{k} = -\bm{r}}^{\bm{r}} c_{\bm{k}}C_{\bm{n}}^{\bm{k}} \bigg) = \left\{c\left(\frac{2\pi \bm{j}}{\bm{n}}\right): \bm{j} = \bm{0},\ldots,\bm{n}-\bm{1} \right\}. \]
\end{theorem}
\noindent Let $\mathcal{C}_{\bm{n}}$ be the set of all $d$-level circulant matrices of size $N(\bm{n}) \times N(\bm{n})$. In light of Theorem \ref{theorem Circulant matrices} we can see that the set $\mathcal{C}_{\bm{n}}$, together with  matrix addition and multiplication, is a commutative ring. For more about circulant matrices see \cite{Davis_Circulant}.
\par A very important notion of the theory of GLT sequences is that of the approximating class of sequences, which will be denoted as a.c.s.. In particular, it is very common in practice to approximate a ``difficult" matrix-sequence by an ``easier" sequence of matrix-sequences, which has the same asymptotic singular value or eigenvalue distribution. For example, such an ``easier" sequence can be used to construct effective preconditioners inside a suitable Krylov subspace method. For the rest of this manuscript, given a matrix $X$, we denote its spectral norm by $\|X\|$. 
\begin{definition} \label{definition a.c.s.}
Let $\{A_n\}_n$ be a matrix-sequence, with $A_n$ of size $d_n \times d_n$, and let $\{\{B_{n,m}\}_n\}_m$ be a sequence of matrix-sequences, with $B_{n,m}$ of size $d_n \times d_n$. We say that $\{\{B_{n,m}\}_n\}_m$ is an approximating class of sequences (a.c.s.) for $\{A_n\}_n$ if for every $m$, there exists $n_m$ such that, for all $n \geq n_m$, we can write:
$$A_n = B_{n,m} + R_{n,m} + N_{n,m},\ \ \textnormal{rank}(R_{n,m}) \leq c(m) d_n,\ \ \|N_{n,m}\| \leq \omega(m),$$
\noindent where $n_m,\ c(m)$, and $\omega(m)$ depend only on $m$, and are such that:
$$\lim_{m\rightarrow \infty} c(m) = \lim_{m \rightarrow \infty} \omega_m = 0.$$
\noindent In that case, we write $\{B_{n,m}\}_n \xrightarrow{\textnormal{a.c.s.}} \{A_n\}_n$.
\end{definition}

\noindent Below, we provide a result, as reported in \cite[Theorem 2.9]{GLT_theory_VOL_2}, which will be very useful when constructing suitable preconditioners later in this paper.
\begin{theorem} \label{theorem a.c.s. algebra}
Let two matrix-sequences $\{A_n\}_n,\ \{A'_n\}_n$ be given, with $A_n,\ A_n'$ of size $d_n \times d_n$, and suppose that $\{B_{n,m}\}_n \xrightarrow{\textnormal{a.c.s.}} \{A_n\}_n$ and $\{B'_{n,m}\}_n \xrightarrow{\textnormal{a.c.s.}} \{A'_n\}_n$. The following properties hold:
\begin{enumerate}
\item $\{B_{n,m}^*\}_n \xrightarrow{\textnormal{a.c.s.}} \{A_n^*\}_n$.
\item $\{c_1 B_{n,m} + c_2 B'_{n,m}\}_n  \xrightarrow{\textnormal{a.c.s.}} \{c_1 A_n + c_2 A'_n\}_n$, for all $c_1,\ c_2 \in \mathbb{C}$.
\item If $\{A_n\}_n$ and $\{A'_n\}_n$ are s.u., then $\{B_{n,m} B'_{n,m}\}_n  \xrightarrow{\textnormal{a.c.s.}} \{A_n A'_n\}_n$.
\item Suppose $\{A_n\}_n$ is s.v.. If $\{B_{n,m}\}_n \xrightarrow{\textnormal{a.c.s.}} \{A_n\}_n$ then $\{B^{\dagger}_{n,m}\}_n \xrightarrow{\textnormal{a.c.s.}} \{A^{\dagger}_n\}_n$.
\end{enumerate}
\end{theorem}
\par All the previous definitions are used to define the notion of Locally Toeplitz (LT) sequences, which in turn are generalized to define the notion of GLT sequences. We briefly define this class of matrix-sequences here, and refer the reader to \cite{GLT_theory_VOL_1,GLT_theory_VOL_2} for a complete derivation of this class, and a vast amount of results concerning sequences belonging in the GLT class. This theory was originally developed in \cite{Tilli_LT}. 
\begin{definition} \label{def: LT operator}
\noindent Let $m,\ n \in \mathbb{N}$, $v:[0,1]\mapsto \mathbb{C}$, and $f \in L^1([-\pi,\pi])$. The 1-level locally Toeplitz operator is defined as the following $n \times n$ matrix:
\begin{equation*}
\begin{split}
LT_n^m(v,f) = &\ \bigg(D_m(v) \otimes T_{\lfloor n/m \rfloor}(f)\bigg) \oplus O_{n \bmod{m}},
\end{split}
\end{equation*}
\noindent where $D_m(v)$ is a diagonal sampling matrix generated by $v$, $T_{\lfloor n/m \rfloor}(f)$ a Toeplitz matrix generated by $f$, and $O_{n \bmod{m}}$ a zero matrix. Let also $\bm{m},\ \bm{n} \in \mathbb{N}^d$, $v:[0,1]^d\mapsto \mathbb{C}$, and $f \in L^1([-\pi,\pi]^d)$. The $d$-level locally Toeplitz operator is recursively defined as the following $N(\bm{n}) \times N(\bm{n})$ matrix:
\begin{equation*}
\begin{split}
LT_{\bm{n}}^{\bm{m}}(v,f_1 \otimes \ldots \otimes f_d) = LT_{n_1,\ldots,n_d}^{m_1,\ldots,m_d} \big( v(x_1,\ldots,x_d),f_1 \otimes \ldots \otimes f_d\big).
\end{split}
\end{equation*}
\end{definition}
\noindent Definition \ref{def: LT operator} allows us to recall the notion of a multilevel locally Toeplitz sequence.
\begin{definition}
Let $\{A_{\bm{n}}\}_n$ be a $d$-level matrix-sequence, let $v:[0,1]^d \mapsto \mathbb{C}$ be Riemann-integrable and let $f \in L^1([-\pi,\pi]^d)$. We say that $\{A_{\bm{n}}\}_n$ is a ($d$-level) locally Toeplitz sequence with symbol $v \otimes f$, and we write $\{A_{\bm{n}}\}_n \sim_{LT} v \otimes f$, if:
\[\{LT_{\bm{n}}^{\bm{m}}(v,f)\}_n \xrightarrow{\textnormal{a.c.s.}}\{A_{\bm{n}}\}_n,\ \textnormal{as }\bm{m} \rightarrow \infty.\]
\end{definition}
\par We are now able to define generalized locally Toeplitz sequences.
\begin{definition} 
Let a $d$-level matrix-sequence $\{A_{\bm{n}}\}_n$, and a measurable function $\kappa:[0,1]^d\times[-\pi,\pi]^d \mapsto \mathbb{C}$ be given. Suppose that $\forall\ \epsilon>0$ there exists a finite number of $d$-level LT sequences $\{A_{\bm{n}}^{(i,\epsilon)}\}_n\sim_{LT} v_{i,\epsilon} \otimes f_{i,\epsilon}$, $i = 1,\ldots,N_{\epsilon}$, such that as $\epsilon \rightarrow 0$:
 \[\sum_{i = 1}^{N_{\epsilon}} v_{i,\epsilon} \otimes f_{i,\epsilon} \rightarrow \kappa\textnormal{ in measure, and } \big\{ \sum_{i=1}^{N_{\epsilon}}A_{\bm{n}}^{(i,\epsilon)}\big\}_n \xrightarrow{\textnormal{a.c.s.}} \{A_{\bm{n}}\}_n.\]
\noindent Then $\{A_{\bm{n}}\}_n$ is a $d$-level GLT sequence with symbol $\kappa$, and we write $\{A_{\bm{n}}\}_n \sim_{GLT} \kappa$.
\end{definition}
\par The GLT class contains a wide range of matrix-sequences arising from various discretization methods of numerous differential equations. In the following Theorem we present some important properties of GLT sequences that will be used later in this paper. This is only a subset of the properties of multilevel GLT sequences, and the reader is referred to \cite{GLT_theory_VOL_1,GLT_theory_VOL_2} for a detailed derivation of all the results presented in this section. Given a measurable function $\kappa$, we denote its complex conjugate by $\bar{\kappa}$.
\begin{theorem} \label{theorem GLT properties}
Let $\{A_{\bm{n}}\}_n$ and $\{B_{\bm{n}}\}_n$ be two d-level matrix-sequences and $\kappa,\ \xi: [0,1]^d \times [-\pi,\pi]^d \mapsto \mathbb{C}$ two measurable functions. Assume that $\{A_{\bm{n}}\}_n$ is a GLT sequence with symbol $\kappa$, while $\{B_{\bm{n}}\}_n$ a GLT sequence with symbol $\xi$. Then:
\begin{enumerate}
\item  If $A_{\bm{n}}$ are Hermitian then $\{A_{\bm{n}}\}_n \sim_{\lambda} \kappa$.
\item $\{A^*_{\bm{n}}\}_n \sim_{GLT} \bar{\kappa}$.
\item $\{c_1 A_{\bm{n}} + c_2 B_{\bm{n}}\}_n \sim_{GLT} c_1 \kappa + c_2 \xi,\ \textnormal{for all } c_1,\ c_2 \in \mathbb{C}$.
\item $\{A_{\bm{n}}B_{\bm{n}}\}_n \sim_{GLT} \kappa \xi$.
\item If $\kappa \neq 0$ almost everywhere, then $\{A^{\dagger}_{\bm{n}}\}_n \sim_{GLT} \kappa^{-1}$.
\item Let a sequence of d-level matrix-sequences $\{B_{\bm{n},m}\}_n \sim_{GLT} \kappa_m$. Then, we have that $\{B_{\bm{n},m}\}_n \xrightarrow{\textnormal{a.c.s.}} \{A_{\bm{n}}\}_n$ if and only if $\kappa_m \rightarrow \kappa$ in measure.
\item If $\{A_{\bm{n}}\}_n \sim_{GLT} \kappa$ and each 	$A_{\bm{n}}$ is Hermitian, then $\{f(A_{\bm{n}})\}_n \sim_{GLT} f(\kappa)$ for every continuous function $f : \mathbb{C} \mapsto \mathbb{C}$.
\end{enumerate}
\end{theorem}

\section{A Structure Preserving Method} \label{Section A structure preserving method}

\par In this section, we will derive an optimization method suitable for solving convex quadratic optimization problems, with linear constraints arising from the discretization of some continuous process. The assumption on the constraints is that the generated (multilevel) matrix-sequence is a GLT sequence. Let us consider the following generic Differential Equation (DE):
\[\mathrm{D} \mathrm{y}(\bm{x},t) = \mathrm{g}(\bm{x},t),\]
\noindent where $\mathrm{D}$ denotes some linear differential operator associated with the DE, $\bm{x}$ is a $(d-1)$-dimensional spatial variable and $t \geq 0$ is the time variable. Since analytical solutions are not readily available for various differential operators, we discretize the previous equation given an arbitrary numerical method, and instead solve a sequence of linear systems of the form:
\begin{equation} \label{discrete arbitrary DE}
\{D_{\bm{n}} y_{\bm{n}}\}_n = \{g_{\bm{n}}\}_n,
\end{equation}
\noindent with size $d_n \times d_n$ and $d_n = N(\bm{n})$, such that $\bm{n} \rightarrow \infty$ as $n \rightarrow \infty$.  
\par Concerning the objective of the studied model, we assume that it may be summarized by a convex functional $\mathrm{J_1}(\mathrm{y}(\bm{x},t))$. Usually, such a functional measures the misfit between the state $\mathrm{y}(\bm{x},t)$ and some desired state $\mathrm{\bar{y}}(\bm{x},t)$, and we will focus our attention on this class of (inverse) problems. In other words, we expect that the discretized version of this functional will be of the form $\frac{1}{2}(y-\bar{y})^*J_1(y-\bar{y})$, with $J_1$ a symmetric positive (semi-)definite matrix. As is common in such problems, the linear systems in \eqref{discrete arbitrary DE} usually admit more than one solution and hence a regularization functional is usually employed to guarantee that the chosen solution will have some desired properties, depending on the initial DE under consideration. In other words, we introduce a control variable $\mathrm{u}(\bm{x},t)$ which is linked to the state variable as follows:
\[\mathrm{D} \mathrm{y}(\bm{x},t) + \mathrm{u}(\bm{x},t) = \mathrm{g}(\bm{x},t).\]
\noindent The size of the control is measured using some convex functional $\mathrm{J_2}(\mathrm{u}(\bm{x},t))$.
\par Finally, we allow further restrictions on the state and control variables in the form of inequality constraints (which depend on the problem under consideration). By combining all the previous, we obtain the following generic model that is studied in this paper:
\begin{equation} \label{generic inverse problem}
\begin{split}
\min_{\mathrm{y},\mathrm{u}} \       &\ \mathrm{J}(\mathrm{y}(\bm{x},t),\mathrm{u}(\bm{x},t)) =  \mathrm{J_1}(\mathrm{y}(\bm{x},t)) +  \mathrm{J_2}(\mathrm{u}(\bm{x},t)) \\
\text{s.t.}\ &\ \mathrm{D} \mathrm{y}(\bm{x},t) + \mathrm{u}(\bm{x},t) = \mathrm{g}(\bm{x},t),\\
        \    &\ \mathrm{y_{a}}(\bm{x},t) \leq \mathrm{y}(\bm{x},t)  \leq \mathrm{y_{b}}(\bm{x},t),\quad \mathrm{u_{a}}(\bm{x},t) \leq \mathrm{u}(\bm{x},t)  \leq \mathrm{u_{b}}(\bm{x},t).
\end{split}
\end{equation}
\noindent The problem is considered on a given compact space-time domain $\Omega \times (0,T)$, for some $T > 0$, where $\Omega \subset \mathbb{R}^{d-1}$ has boundary $\partial \Omega$. The algebraic inequality constraints are assumed to hold a.e. on $\Omega \times(0,T)$. We further note that the restrictions ${\rm y_a},\ {\rm y_b},\ {\rm u_a}$, and ${\rm u_b}$ may take the form of constants, or functions in spatial and/or temporal variables. The boundary conditions are not specified since they do not affect the analysis in this section. Notice that problem \eqref{generic inverse problem} includes the case of equality-constrained optimization, by allowing unbounded restriction functions.
\par We discretize problem \eqref{generic inverse problem}, using an arbitrary numerical method, to find an approximate solution by solving a sequence of optimization problems of the form:
\begin{equation} \label{discretized generic inverse problem}
\begin{split}
\min_{y_{\bm{n}},u_{\bm{n}}} \       &\ \bigg(\frac{1}{2}(y_{\bm{n}}-\bar{y}_{\bm{n}})^* J_{1_{\bm{n}}} (y_{\bm{n}}-\bar{y}_{\bm{n}}) + \frac{1}{2}u_{\bm{n}}^*J_{2_{\bm{n}}} u_{\bm{n}}\bigg) \\
\text{s.t.}\ &\ D_{\bm{n}} y_{\bm{n}} + u_{\bm{n}} = g_{\bm{n}},\\
        \    &\ y_{a_{\bm{n}}} \leq y_{\bm{n}}  \leq y_{b_{\bm{n}}},\quad u_{a_{\bm{n}}} \leq u_{\bm{n}}  \leq u_{b_{\bm{n}}},
\end{split}
\end{equation}
\noindent in which the associated matrices are of size $d_n \times d_n$, where $d_n = N(\bm{n})$ and $d_n \rightarrow \infty$ as $n \rightarrow \infty$. Notice that we only assume $J_{1_{\bm{n}}}$ and $J_{2_{\bm{n}}}$ to be symmetric positive semi-definite. Hence, the presented methodology is applicable to a wide range of convex quadratic programming problems. An entry $n_j$ of the multi-index $\bm{n}$ corresponds to the number of discretization points along dimension $j$, with $j \in \{1,\ldots,d\}$, where $n_d$ corresponds to  the time dimension. Below, we summarize our assumptions for the associated matrices in problem \eqref{discretized generic inverse problem}.
\begin{assumption} \label{Assumption GLT matrices}
Given the sequence of problems in \eqref{discretized generic inverse problem}, we assume that:
\begin{itemize}
\item The sequence $\{D_{\bm{n}}\}_n$ is a $d$-level matrix sequence with spectral norm uniformly bounded with respect to $n$, i.e. there exists a constant $C_D$ such that $\|D_{\bm{n}}\| \leq C_D$ for all $n$. Furthermore, there exists a measurable function $\kappa:[0,1]^d\times[-\pi,\pi]^d \mapsto \mathbb{C}$, which is the symbol of $\{D_{\bm{n}}\}_n$, so that $\{D_{\bm{n}}\}_n \sim_{GLT} \kappa$.
\item The sequences $\{J_{1_{\bm{n}}}\}_n$ and $\{J_{2_{\bm{n}}}\}_n$ are two $d$-level matrix sequences, with uniformly bounded spectral norms with respect to $n$. Furthermore, there exist two measurable functions $\xi_1,\ \xi_2:[0,1]^d \times [-\pi,\pi]^d\mapsto \mathbb{R}$, such that $\xi_1 \geq 0$, $\xi_2 \geq 0$, $\{J_{1_{\bm{n}}}\}_n \sim_{GLT} \xi_1$, and $\{J_{2_{\bm{n}}}\}_n \sim_{GLT} \xi_2$.
\end{itemize}
\end{assumption}
\noindent We note that a wide range of numerical discretizations of DEs satisfy this assumption (see \cite{GLT_theory_VOL_1,GLT_theory_VOL_2} for a plethora of applications). Notice also that the requirement that $\xi_1$ and $\xi_2$ are real and non-negative follows from the positive semi-definiteness of $J_{1_{\bm{n}}}$ and $J_{2_{\bm{n}}}$. Towards the end of this section we discuss how one could still apply the presented methodology successfully without requiring the GLT structure of the discretized objective function (i.e. by requiring only boundedness and convexity).
\par Before presenting the proposed optimization method for solving problems of the form \eqref{discretized generic inverse problem}, we note a negative result concerning a large class of optimization methods. More specifically, problems like \eqref{discretized generic inverse problem} are often solved using an Interior Point Method (IPM), or some Active-Set (AS) type of method. However, such problems are usually highly structured, and this structure must be exploited, given that the problem size increases indefinitely as one refines the discretization. Obviously, any AS method would fail in maintaining the structure, as only a subset of the constraints of \eqref{discretized generic inverse problem} is considered at each AS iteration and hence the structure of the AS sub-problems will be unknown. In fact, any optimization method whose sub-problems arise by projecting the variables of the problem in a subspace would face this issue. 
\par On the other hand, IPMs deal with the inequality constraints by introducing logarithmic barriers in the objective (see for example \cite{paper_23}). Then, at every IPM iteration, one forms the optimality conditions of the barrier sub-problem, and approximately solves them using Newton's method. If we assume that there exists a symbol $f$ which describes the asymptotic eigenvalue distribution of the sequence of Hessian matrices of the logarithmic barrier, we arrive at a contradiction. Indeed, the sequence of Hessian matrices arising from the logarithmic barriers introduced by IPM are not s.u.. This in turns contradicts the assumption that $f$ is the symbol of this matrix sequence, since if an arbitrary matrix sequence is such that $\{L_n\}_n \sim_{\sigma} f$, for some measurable function $f$, then $\{L_n\}_n$ must be s.u. (see \cite[Section 9--S1]{GLT_theory_VOL_1}). In particular, any GLT sequence is s.u., and hence the sequence of Hessian matrices of the logarithmic barrier functions cannot be a GLT sequence. As a consequence, the system matrix of the optimality conditions of each barrier sub-problem, within the IPM, will not be in the GLT class.
\subsection{Alternating Direction Method of Multipliers}
\par In order to overcome the previous issues, we propose the use of an alternating direction method of multipliers (see \cite[Section 5]{paper_26} and the references therein), which separates the equality from the inequality constraints, thus allowing us to preserve the structure found in the matrices associated with \eqref{discretized generic inverse problem}. We should mention here that while ADMM allows us to retain the underlying structure of the problem, it comes at a cost. It is well-known (see e.g. \cite{paper_26}) that ADMM leads to relatively slow convergence and hence is not suitable for finding very accurate solutions. Nevertheless, a 4-digit accurate solution can generally be found in reasonable CPU time. Furthermore, the linear system solved at each ADMM iteration does not change, and hence, if a suitable preconditioner exploiting the problem structure is found, it only needs to be computed once. Finally, linear convergence can also be shown, under certain assumptions on the problem under consideration (such as strong convexity, see \cite{paper_53}).
\par We begin by rewriting problem \eqref{discretized generic inverse problem}, after introducing some auxiliary variables $z_{y_{\bm{n}}},z_{u_{\bm{n}}}$ of size $N(\bm{n})$:
\begin{equation} \label{Generic ADMM Full problem}
\begin{split}
\min_{y_{\bm{n}},u_{\bm{n}},z_{y_{\bm{n}}},z_{u_{\bm{n}}}} & \ \bigg(\frac{1}{2}(y_{\bm{n}}-\bar{y}_{\bm{n}})^* J_{1_{\bm{n}}} (y_{\bm{n}}-\bar{y}_{\bm{n}}) + \frac{1}{2}u_{\bm{n}}^*J_{2_{\bm{n}}} u_{\bm{n}}\bigg) \\
 \text{s.t.}\ \  & \  D_{\bm{n}}y_{\bm{n}} +   u_{\bm{n}} = g_{\bm{n}}, \\ 
  &\ y_{\bm{n}} = z_{y_{\bm{n}}} , \  u_{\bm{n}} = z_{u_{\bm{n}}},\\
 &\  y_{a_{\bm{n}}} \leq z_{y_{\bm{n}}} \leq y_{b_{\bm{n}}},\ u_{a_{\bm{n}}} \leq z_{u_{\bm{n}}} \leq u_{b_{\bm{n}}}.
\end{split}
\end{equation}
\noindent Next, we define the augmented Lagrangian function corresponding to (\ref{Generic ADMM Full problem}):
\begin{equation} \label{Generic ADMM aug Lagrangian}
\begin{split}
\mathcal{L}_{\delta}(y_{\bm{n}},u_{\bm{n}},z_{y_{\bm{n}}},&z_{u_{\bm{n}}},p_{_{\bm{n}}},w_{y_{\bm{n}}},w_{u_{\bm{n}}}) =\ \frac{1}{2}(y_{\bm{n}}-\bar{y}_{\bm{n}})^* J_{1_{\bm{n}}} (y_{\bm{n}}-\bar{y}_{\bm{n}}) + \frac{1}{2} u_{\bm{n}}^* J_{2_{\bm{n}}} u_{\bm{n}} \\
 &+\ p_{\bm{n}}^*(D_{\bm{n}}y_{\bm{n}} + u_{\bm{n}}- g_{\bm{n}}) +  w_{y_{\bm{n}}}^*(y_{\bm{n}}-z_{y_{\bm{n}}}) +  w_{u_{\bm{n}}}^*(u_{\bm{n}} - z_{u_{\bm{n}}}) \\
 &+\ \frac{1}{2\delta}\left(\|D_{\bm{n}}y_{\bm{n}}+u_{\bm{n}}-g_{\bm{n}}\|_2^2 + \|y_{\bm{n}}-z_{y_{\bm{n}}}\|_2^2 + \|u_{\bm{n}}-z_{u_{\bm{n}}}\|_2^2 \right), 
\end{split}
\end{equation}
\noindent where $p_{\bm{n}},\ w_{y_{\bm{n}}}$, and $w_{u_{\bm{n}}}$ are the dual variables corresponding to each of the equality constraints of (\ref{Generic ADMM Full problem}). An ADMM applied to model (\ref{Generic ADMM Full problem}) is given in Algorithm \ref{ADMM algorithm}. We omit specific details of the algorithm. The reader is referred to \cite{paper_26} for a basic proof of convergence of Algorithm \ref{ADMM algorithm}, as well as a detailed overview of ADMM. For a convergence proof for the case where complex variables and matrices are allowed, the reader is referred to \cite{ADMMComplex}. Linear convergence of a generalization of this algorithm, under certain assumptions, can be found in \cite{paper_53} and the references therein. We should note that the step-length $\rho$ in \eqref{ADMM dual update 1} and \eqref{ADMM dual update 2 and 3} plays an important role in the convergence behavior of ADMM, and in fact, convergence of Algorithm \ref{ADMM algorithm} is guaranteed for any $\rho  \in (0,\frac{\sqrt{5}+1}{2})$ (see \cite{book_5}).
\begin{algorithm}[!ht]
\caption{(2-Block) Standard ADMM }
    \label{ADMM algorithm}
    \textbf{Input:}  Let $y_{\bm{n}}^0,u_{\bm{n}}^0,z_{y_{\bm{n}}}^0,z_{u_{\bm{n}}}^0,p_{\bm{n}}^0,w_{y_{\bm{n}}}^0,w_{u_{\bm{n}}}^0 \in \mathbb{C}^{N(\bm{n})}$, $\delta > 0$, $\rho  \in (0,\frac{\sqrt{5}+1}{2})$. 
\begin{algorithmic}
\For {($j = 0,1,\dotsc$)}
 \begin{subequations} 
\begin{align}
(y_{\bm{n}}^{j+1},u_{\bm{n}}^{j+1}) &=\  \underset{y_{\bm{n}},u_{\bm{n}}}{\arg\min}\big\{\mathcal{L}_{\delta}(y_{\bm{n}},u_{\bm{n}},z^{j}_{y_{\bm{n}}},z^{j}_{u_{\bm{n}}},p^j_{\bm{n}},w^j_{y_{\bm{n}}},w^j_{u_{\bm{n}}}) \big\}  \label{ADMM subproblem 1}\\
(z_{y_{\bm{n}}}^{j+1},z_{u_{\bm{n}}}^{j+1}) &=\  \underset{z_y \in [y_{a},y_{b}],\ z_u \in [u_{a},u_{b}]}{\arg\min} \big \{ \mathcal{L}_{\delta}(y_{\bm{n}}^{j+1},u_{\bm{n}}^{j+1},z_{y_{\bm{n}}},z_{u_{\bm{n}}},p_{\bm{n}}^j,w_{y_{\bm{n}}}^j,w_{u_{\bm{n}}}^j) \big\}   \label{ADMM subproblem 2}\\ 
p_{\bm{n}}^{j+1} &= \ p_{\bm{n}}^j + \frac{\rho}{\delta}(D_{\bm{n}}y_{\bm{n}}^{j+1}+ u_{\bm{n}}^{j+1} - g_{\bm{n}})\label{ADMM dual update 1}\\
(w_{y_{\bm{n}}}^{j+1},w_{u_{\bm{n}}}^{j+1}) &=\ \left(w_{y_{\bm{n}}}^j+ \frac{\rho}{\delta}(y_{\bm{n}}^{j+1}-z_{y_{\bm{n}}}^{j+1}), w_{u_{\bm{n}}}^j + \frac{\rho}{\delta}(u_{\bm{n}}^{j+1}-z_{u_{\bm{n}}}^{j+1})\right) \label{ADMM dual update 2 and 3}
\end{align}
\end{subequations}
\EndFor
\end{algorithmic}
\end{algorithm}
\par One can easily observe that the most challenging step of Algorithm \ref{ADMM algorithm}, is that of solving (\ref{ADMM subproblem 1}). The optimality conditions of (\ref{ADMM subproblem 1}), at iteration $j$, read as follows:
\begin{equation}\label{generic ADMM subproblem 1 optimality conditions}
\begin{split}
\begin{bmatrix}
J_{1_{\bm{n}}}  + \frac{1}{\delta}(D_{\bm{n}}^* D_{\bm{n}} +  I_{\bm{n}})  & \frac{1}{\delta}D_{\bm{n}}^* \\
 \frac{1}{\delta}D_{\bm{n}}  &  J_{2_{\bm{n}}}   + \frac{2}{\delta}I_{\bm{n}}
\end{bmatrix} \begin{bmatrix} y_{\bm{n}}\\ u_{\bm{n}} \end{bmatrix}  = \begin{bmatrix}  
\eta_1\\
\eta_2
 \end{bmatrix},
 \end{split}
\end{equation}
\noindent where 
\[ \eta_1 = J_{1_{\bm{n}}}\bar{y}_{\bm{n}} - D_{\bm{n}}^*p_{\bm{n}}^j - w_{y_{\bm{n}}}^j + \frac{1}{\delta} (D_{\bm{n}}^*g_{\bm{n}} +  z_{y_{\bm{n}}}^{j}),\ \  \eta_2 = - p_{\bm{n}}^j - w_{u_{\bm{n}}}^j +\frac{1}{\delta}(g_{\bm{n}} + z_{u_{\bm{n}}}^{j}).\]

\par Solving (\ref{generic ADMM subproblem 1 optimality conditions}) directly is not a good idea in our case, since its coefficient matrix is not expected to be cheap or convenient to work with. Instead, we can merge steps (\ref{ADMM subproblem 1}) and (\ref{ADMM dual update 1}) to obtain a more flexible saddle point system. More specifically, to take (\ref{ADMM dual update 1}) into account, we substitute  $p_{\bm{n}} = p_{\bm{n}}^{j} + \frac{\rho}{\delta} (D_{\bm{n}}y_{\bm{n}} + u_{\bm{n}} - g_{\bm{n}})$
\noindent into (\ref{generic ADMM subproblem 1 optimality conditions}), and the optimality conditions of (\ref{ADMM subproblem 1}) and (\ref{ADMM dual update 1}) can then be written as:
\begin{equation} \label{generic ADMM subproblems 1and3 optimality conditions}
\begin{split}
\begin{bmatrix}
\rho (J_{1_{\bm{n}}} + \frac{1}{\delta} I_{\bm{n}}) & 0 & D_{\bm{n}}^*\\
0 & \rho(J_{2_{\bm{n}}} + \frac{1}{\delta}I_{\bm{n}}) &  I_{\bm{n}}\\
D_{\bm{n}} &  I_{\bm{n}} & -\frac{\delta}{\rho} I_{\bm{n}}
\end{bmatrix} \begin{bmatrix}
y_{\bm{n}}\\
u_{\bm{n}}\\
p_{\bm{n}}
\end{bmatrix}  = \\ 
\qquad \begin{bmatrix}
\rho (J_{1_{\bm{n}}} \bar{y}_{\bm{n}} -  w_{y_{\bm{n}}}^j + \frac{1}{\delta}z_{y_{\bm{n}}}^{j}) + (1-\rho) D_{\bm{n}}^* p_{\bm{n}}^j\\
\rho(- w_{u_{\bm{n}}}^j + \frac{1}{\delta}z_{u_{\bm{n}}}^{j}) + (1-\rho)p_{\bm{n}}^j\\
 g_{\bm{n}} - \frac{\delta}{\rho} p_{\bm{n}}^j
\end{bmatrix}. 
\end{split}
\end{equation}
\noindent At this point, we have to decide how to solve (\ref{generic ADMM subproblems 1and3 optimality conditions}). For simplicity of exposition, we present here only one way of solving system (\ref{generic ADMM subproblems 1and3 optimality conditions}), by forming the normal equations and then employing the Preconditioned Conjugate Gradient method (PCG) to solve the resulting positive definite system, assuming that its $(2,2)$ block will be easily invertible. We note that the developments in this section hold for any Schur complement of the matrix in \eqref{generic ADMM subproblems 1and3 optimality conditions} (the choice of which Schur complement to use heavily depends on the problem under consideration). The case where neither the $(1,1)$ nor the $(2,2)$ block is easily invertible,  will be treated at the end of this section. Pivoting the second and then the third block equation of this system, yields: 
\begin{equation*}
\begin{split}
u_{\bm{n}} = &\left(\rho\left( J_{2_{\bm{n}}} + \frac{1}{\delta}I_{\bm{n}}\right)\right)^{-1}\left(- p_{\bm{n}} - \rho w_{u_{\bm{n}}}^j + \frac{\rho}{\delta}z_{u_{\bm{n}}}^{j} + (1-\rho)p_{\bm{n}}^j\right),\\
p_{\bm{n}} = &\left( \left( \rho J_{2_{\bm{n}}} + \frac{\rho}{\delta}I_{\bm{n}}\right)^{-1} + \frac{\delta}{\rho} I_{\bm{n}} \right)^{-1} \left(D_{\bm{n}}y_{\bm{n}}+r \right), \\
 r =  &- g_{\bm{n}} + \frac{\delta}{\rho} p_{\bm{n}}^j -   \left( \rho\left(J_{2_{\bm{n}}} + \frac{1}{\delta}I_{\bm{n}}\right)\right)^{-1}\left(\rho\left(- w_{u_{\bm{n}}}^j + \frac{1}{\delta}z_{u_{\bm{n}}}^{j}\right) + (1-\rho)p_{\bm{n}}^j\right), 
 \end{split}
\end{equation*}
\noindent and the resulting normal equations read as follows:
\begin{equation} \label{generic ADMM normal equations}
\begin{split}
S_{\bm{n}} y_{\bm{n}} \coloneqq  \left(\rho \left(J_{1_{\bm{n}}} + \frac{1}{\delta} I_{\bm{n}}\right) + D_{\bm{n}}^* \left(\left(\rho J_{2_{\bm{n}}} + \frac{\rho}{\delta}I_{\bm{n}}\right)^{-1} + \frac{\delta}{\rho} I_{\bm{n}}  \right)^{-1} D_{\bm{n}} \right)y_{\bm{n}} = \\ \rho \left(J_{1_{\bm{n}}} \bar{y}_{\bm{n}} -  w_{y_{\bm{n}}}^j + \frac{1}{\delta}z_{y_{\bm{n}}}^{j}\right) + (1-\rho) D_{\bm{n}}^* p_{\bm{n}}^j -  D_{\bm{n}}^*\left( \left(\rho J_{2_{\bm{n}}} + \frac{\rho}{\delta}I_{\bm{n}}\right)^{-1} + \frac{\delta }{\rho}I_{\bm{n}} \right)^{-1} r.
\end{split}
\end{equation}
\noindent Finally, we should mention that problem (\ref{ADMM subproblem 2}) of Algorithm \ref{ADMM algorithm} is trivial, as it admits a closed form solution. More specifically, we perform the optimization by ignoring the box constraints and then projecting the solution onto the box.
\par In what follows, using Assumption \ref{Assumption GLT matrices}, we present some results concerning the asymptotic behavior of the matrix sequence $\{S_{\bm{n}}\}_n$, by making use of the Theorems presented in the previous section. The latter is produced by refining an arbitrary discretization applied to \eqref{generic inverse problem} (assuming it satisfies Assumption \ref{Assumption GLT matrices}), employing ADMM to the discretized problem, and forming a certain Schur complement of the joint optimality conditions of \eqref{ADMM subproblem 1} and \eqref{ADMM dual update 1}. The solution of \eqref{generic ADMM normal equations} delivers the solution to \eqref{generic ADMM subproblems 1and3 optimality conditions}, and the remaining ADMM sub-problems can be trivially solved in $O(N(\bm{n}))$ operations. Following practical applications, we assume $\delta$ and $\rho$ to be $\Theta(1)$ and constant along the iterations of ADMM (usually $\delta \in [0.01,100]$ and $\rho \in [1,1.618]$) .

\begin{theorem} \label{theorem GLT preserving method}
Given Assumption \textnormal{\ref{Assumption GLT matrices}}, and the sequence $\{S_{\bm{n}}\}_n$, with $S_{\bm{n}}$ given in \eqref{generic ADMM normal equations}, we have that there exists a measurable function $\tau : [0,1]^d\times[-\pi, \pi]^d \mapsto \mathbb{R}$ such that $\tau \geq 0$, $\tau \neq 0$ a.e., and $\{S_{\bm{n}}\}_n \sim_{GLT} \tau$. Moreover, $S_{\bm{n}}$ are Hermitian positive definite, $\{S_{\bm{n}}\}_n \sim_{\lambda} \tau$, and $\{S^{-1}_{\bm{n}}\}_n \sim_{\lambda} \tau^{-1}$.
\end{theorem}
\begin{proof}
Let Assumption \ref{Assumption GLT matrices} hold. Then, we have that there exist three measurable functions $\kappa,\ \xi_1,\ \xi_2:[0,1]^d \times [-\pi,\pi]^d\mapsto \mathbb{C}$, such that $\xi_1 \geq 0$, $\xi_2 \geq 0$, $\{J_{1_{\bm{n}}}\}_n \sim_{GLT} \xi_1$, $\{J_{2_{\bm{n}}}\}_n \sim_{GLT} \xi_2$, and $\{D_{\bm{n}}\}_n \sim_{GLT} \kappa$. Furthermore, we can notice that, for any constant $C > 0$, $\{C I_{\bm{n}}\}_n \sim_{GLT} C$, where $C$ can be considered as a positive constant function (e.g. as a constant on the domain $[-\pi,\pi]^d$, generating a diagonal Toeplitz matrix). This, combined with Theorem \ref{theorem GLT properties} (conditions (2.)--(4.)), yields that:
\begin{equation*}
\begin{split}
\{M_{1_{\bm{n}}}\}_n \coloneqq \bigg\{\rho\left(J_{1_{\bm{n}}} + \frac{1}{\delta} I_{\bm{n}}\right) \bigg\}_n \sim_{GLT} &\ \rho(\xi_1 + \delta^{-1}),\\
\{M_{2_{\bm{n}}}\}_n \coloneqq\bigg\{ \left(\rho\left(J_{2_{\bm{n}}} + \frac{1}{\delta}I_{\bm{n}}\right)\right)^{-1} + \frac{\delta}{\rho} I_{\bm{n}} \bigg\}_n \sim_{GLT} &\ \big(\rho(\xi_2 + \delta^{-1})\big)^{-1} + \frac{\delta}{\rho}.
\end{split}
\end{equation*}
\noindent Similarly, from Theorem \ref{theorem GLT properties} (conditions (2.)--(5.)), we have that:
\begin{equation*}
\bigg\{M_{1_{\bm{n}}} + D_{\bm{n}}^* M_{2_{\bm{n}}}^{-1} D_{\bm{n}} \bigg\}_n \sim_{GLT} \rho(\xi_1 + \delta^{-1}) + |\kappa|^2  \bigg(\big(\rho(\xi_2 + \delta^{-1})\big)^{-1} + \frac{\delta}{\rho} \bigg)^{-1},
\end{equation*}
\noindent where we used that $\bar{\kappa}\kappa = |\kappa|^2$. Setting $\tau = \rho(\xi_1 + \delta^{-1}) + |\kappa|^2  \bigg(\big(\rho(\xi_2 + \delta^{-1})\big)^{-1} + \frac{\delta}{\rho} \bigg)^{-1}$ and noticing that $\tau > 0$ completes the proof.
\end{proof}
\par Subsequently we present some possible approaches that could allow one to take advantage of the structure preserving property of ADMM. In particular, three possible ways of exploiting the preserved structure are discussed here. However, other approaches could be possible. For this analysis, we will make use of the following proposition:
\begin{proposition} \label{Prop. ACS matrices}
Let Assumption \textnormal{\ref{Assumption GLT matrices}} hold. Then, there exist sequences of $d$-level matrix-sequences $\{\{\tilde{D}_{\bm{n},m}\}_n\}_m,$ $\{\{\tilde{J}_{1_{\bm{n},m}}\}_n\}_m,$ $\{\{\tilde{J}_{2_{\bm{n},m}}\}_n\}_m$, with uniformly bounded spectral norms with respect to $n$ and $m$, and sequences of measurable functions $\{\kappa_m\}_m,$ $\{\xi_{1_m}\}_m$, and $\{\xi_{2_m}\}_m$ such that $\kappa_m,\ \xi_{1_m},\ \xi_{2_m}: [0,1]^d\times[-\pi,\pi]^d \mapsto \mathbb{C}$, $\xi_{1_m},\ \xi_{2_m}$ are real a.e., non-negative, and:
\begin{itemize}
\item $\{\tilde{D}_{\bm{n},m}\}_n \xrightarrow{\textnormal{a.c.s.}} \{D_{\bm{n}}\}_n$, $\{\tilde{D}_{\bm{n},m}\}_n \sim_{GLT} \kappa_m$, with $\kappa_m \rightarrow \kappa$ in measure,
\item $\{\tilde{J}_{1_{\bm{n},m}}\}_n \xrightarrow{\textnormal{a.c.s.}} \{J_{1_{\bm{n}}}\}_n$, $\{\tilde{J}_{1_{\bm{n},m}}\}_n \sim_{GLT} \xi_{1_m}$, with $\xi_{1_m} \rightarrow \xi_1$ in measure,
\item $\{\tilde{J}_{2_{\bm{n},m}}\}_n\xrightarrow{\textnormal{a.c.s.}} \{J_{2_{\bm{n}}}\}_n$,  $\{\tilde{J}_{2_{\bm{n},m}}\}_n \sim_{GLT} \xi_{2_m}$, with $\xi_{2_m} \rightarrow \xi_2$ in measure.
\end{itemize}
\end{proposition}
\begin{proof}
The proof can be found in \cite[Theorem 8.6]{GLT_theory_VOL_1} for the unilevel case and \cite[Theorem 5.6]{GLT_theory_VOL_2} for the multilevel case.
\end{proof}
\noindent For the rest of this section we will assume that we have such sequences of $d$-level GLT sequences available, satisfying the conditions stated in Proposition \ref{Prop. ACS matrices}. We further assume that these approximate sequences are comprised of matrices that are easy to compute and invert (whenever possible).
\subsection{Schur complement approximations}
\par In what follows we present various Schur complement approximations that could potentially serve as preconditioners inside PCG, for solving systems of the form of \eqref{generic ADMM normal equations} (or any other Schur complement of system \eqref{generic ADMM subproblems 1and3 optimality conditions}). The viability of each of the following approaches depends on the structure of the problem, as well as the choice of the discretization. We note that the different approaches are presented for completeness, as well as an indicator of the generality of the proposed methodology. In particular, as the convergence behavior of ADMM does not depend on the choice of preconditioner (assuming that the PCG converges to a desired accuracy), we will only make use of one of the following Schur complement approximations when presenting computational results.

\subsubsection{A Schur complement block approximation}
\par Given three sequences of GLT sequences  $\{\{\tilde{D}_{\bm{n},m}\}_n\}_m$, $\{\{\tilde{J}_{1_{\bm{n},m}}\}_n\}_m,$ $\{\{\tilde{J}_{2_{\bm{n},m}}\}_n\}_m$, satisfying the conditions of Proposition \ref{Prop. ACS matrices}, we define the following approximation for the matrix in \eqref{generic ADMM normal equations}:
\begin{equation} \label{first schur complement preconditioner}
\tilde{S}_{\bm{n},m} = \rho\left(\tilde{J}_{1_{\bm{n},m}} + \frac{1}{\delta} I_{\bm{n}}\right) + \tilde{D}_{\bm{n},m}^* \left(\left(\rho \tilde{J}_{2_{\bm{n},m}} + \frac{\rho}{\delta}I_{\bm{n}}\right)^{-1} + \frac{\delta}{\rho} I_{\bm{n}}  \right)^{-1} \tilde{D}_{\bm{n},m}.
\end{equation}

\begin{theorem} \label{theorem first schur complement preconditioner}
Let Assumption \textnormal{\ref{Assumption GLT matrices}} hold, and assume that we have available the sequences $\{\{\tilde{D}_{\bm{n},m}\}_n\}_m,$ $\{\{\tilde{J}_{1_{\bm{n},m}}\}_n\}_m,$ and $\{\{\tilde{J}_{2_{\bm{n},m}}\}_n\}_m$, satisfying the conditions of Proposition \textnormal{\ref{Prop. ACS matrices}}. By defining $\tilde{S}_{\bm{n},m}$ as in \eqref{first schur complement preconditioner}, we have:
\begin{itemize}
\item $\{\tilde{S}_{\bm{n},m}\}_n \xrightarrow{\textnormal{a.c.s.}} \{S_{\bm{n}}\}_n$, $\{\tilde{S}_{\bm{n},m}\}_n \sim_{GLT} \tau_m$, and $\tau_m \rightarrow \tau$ in measure, where $\tau$ is given in Theorem \ref{theorem GLT preserving method} and:
\begin{equation} \label{symbol of first preconditioner} 
\tau_m = \rho(\xi_{1_m} + \delta^{-1}) + |\kappa_m|^2  \bigg(\big(\rho(\xi_{2_m} + \delta^{-1})\big)^{-1} + \frac{\delta}{\rho} \bigg)^{-1}.
\end{equation}
\item The sequence $\{\tilde{S}^{-1}_{\bm{n},m} S_{\bm{n}}\}_n$ is weakly clustered at $1$.
\item For any $n,\ m$, the eigenvalues of $\tilde{S}^{-1}_{\bm{n},m}S_{\bm{n}}$ lie in the interval $\bigg[ \frac{1}{C_s},C_s\bigg]$, where $C_s$ is a positive constant uniformly bounded with respect to $n$ and $m$.
\end{itemize}
\end{theorem}
\begin{proof}
\noindent For the first condition, by Proposition \ref{Prop. ACS matrices} as well as Theorem \ref{theorem a.c.s. algebra} (conditions (1.)--(4.)), we get that $\{\tilde{S}_{\bm{n},m}\}_n \xrightarrow{\textnormal{a.c.s.}} \{S_{\bm{n}}\}_n$. Using Proposition \ref{Prop. ACS matrices} again, this time combined with Theorem \ref{theorem GLT properties} (conditions (2.)--(5.)) and Lemma \ref{lemma convergence in measure}, yields that $\tau_m$ is given by \eqref{symbol of first preconditioner}, and $\tau_m \rightarrow \tau$ in measure.
\par For the second condition, we firstly note that the sequence under consideration is Hermitian and positive definite by Assumption \ref{Assumption GLT matrices}. Then, using that $\{\tilde{S}_{\bm{n},m}\}_n \xrightarrow{\textnormal{a.c.s.}} \{S_{\bm{n}}\}_n$ implies that, for every $m$, there exists $n_m$ such that for all $n \geq n_m$:
\begin{equation} \label{eq:theorem schur compl prec. temp 1}
 S_{\bm{n}} = \tilde{S}_{\bm{n},m} + R_{\bm{n},m} + N_{\bm{n},m},\quad \textnormal{rank}(R_{\bm{n},m}) \leq c(m) d_n,\quad \|N_{\bm{n},m}\| \leq \omega(m),\end{equation}
\noindent where $n_m$, $c(m)$ and $\omega(m)$ depend only on $m$ and are such that:
\[\lim_{m \rightarrow \infty} c(m) = \lim_{m\rightarrow \infty}\omega(m) = 0.\]
\noindent By assumption, it is easy to see that any $\tilde{S}_{\bm{n},m}$ has a spectral norm uniformly bounded in $n$ and $m$. Furthermore, since $\delta =\Theta(1)$ and $\rho = \Theta(1)$, we observe that $\|\tilde{S}^{-1}_{\bm{n},m}\|$ is also uniformly bounded in $n$. Hence, by multiplying both sides of \eqref{eq:theorem schur compl prec. temp 1} by $\tilde{S}^{-1}_{\bm{n},m}$:
\[ \tilde{S}^{-1}_{\bm{n},m} S_{\bm{n}} = I_{\bm{n}} + \tilde{R}_{\bm{n},m} + \tilde{N}_{\bm{n},m},\]
\noindent where $\tilde{R}_{\bm{n},m} = \tilde{S}^{-1}_{\bm{n},m}R_{\bm{n},m}$ (thus $\textnormal{rank}(\tilde{R}_{\bm{n},m}) \leq \textnormal{rank}(R_{\bm{n},m}) \leq c(m) d_n$) and $\tilde{N}_{\bm{n},m} = \tilde{S}^{-1}_{\bm{n},m}N_{\bm{n},m}$ (and hence $\|\tilde{N}_{\bm{n},m}\| \leq C \omega(m)$, for some constant $C > 0$, independent of $m$ and $n$). This, along with the definition of a weak cluster in Definition \ref{definition: clustering}, proves the second condition.
\par For the third condition, let us take some constant $C_{\dagger}$ of $O(1)$, such that:  
\[\max\big\{\|D_{\bm{n}}\|,\|\tilde{D}_{\bm{n},m}\|,\|J_{1_{\bm{n}}}\|,\|J_{2_{\bm{n}}}\|,\|\tilde{J}_{1_{\bm{n},m}}\|,\|\tilde{J}_{2_{\bm{n},m}}\|\big\} \leq C_{\dagger},\ \forall\ n,\ m.\]
\noindent We know that such a constant exists by Assumption \ref{Assumption GLT matrices}. Then, we have that $\lambda_{\min}(S_{\bm{n}}) \geq \frac{\rho}{\delta}$ and $\lambda_{\max}(S_{\bm{n}}) \leq \rho C_{\dagger} + \frac{\rho}{\delta} + \frac{\rho}{\delta}C_{\dagger}^2$, for any $n$. The exact same bounds hold also for $\tilde{S}_{\bm{n},m}$, for every $n$ and $m$. Using these bounds, we can easily show that:
\[ \lambda_{\min}(\tilde{S}^{-1}_{\bm{n},m} S_{\bm{n}}) \geq \frac{1}{C_{\dagger}^2 + \delta C_{\dagger} + 1},\ \ \lambda_{\max}(\tilde{S}^{-1}_{\bm{n},m}S_{\bm{n}}) \leq C_{\dagger}^2 + \delta C_{\dagger} + 1, \]
\noindent for all $n,\ m$. Upon noticing that $\delta = \Theta(1)$ and $\rho = \Theta(1)$, there exists a constant $C_s = C_{\dagger}^2 + \delta C_{\dagger} + 1$ uniformly bounded with respect to $n$, satisfying the third condition of the Theorem.
\end{proof}
\begin{remark}
Notice that in order to obtain a strong clustering at 1, we would have to employ some extra assumptions. In particular, we would have to require that the sequences given in Assumption \textnormal{\ref{Assumption GLT matrices}} are strongly clustered in the essential range of their symbols, which in turn are required to be different from zero a.e.. Furthermore, we would have to assume that the condition in \eqref{eq:theorem schur compl prec. temp 1} is such that $ c(m)d_n = O(1)$.
\end{remark}
\begin{remark} While Assumption \textnormal{\ref{Assumption GLT matrices}} holds for a wide range of problems, and one is able to find easily computable sequences satisfying the conditions in Proposition \textnormal{\ref{Prop. ACS matrices}}, it is not often the case that the preconditioner in \eqref{first schur complement preconditioner} is easy to compute or invert. If this is the case, then Theorem \textnormal{\ref{theorem first schur complement preconditioner}} guarantees that such a preconditioner will provide a weak cluster of the eigenvalues of the preconditioned matrix at $1$. We note that while this is not optimal, it is expected to be good enough. This is because of the penalty parameter introduced by ADMM (i.e. $\delta = \Theta(1)$), which (along with the assumption that the involved matrices are uniformly bounded in $n$) guarantees that the normal equations matrix will be relatively well-conditioned, and hence PCG will converge in a number of iterations independent of the grid size (however, possibly depending on the conditioning of the problem matrix as well as the problem parameters).
\par The use of the preconditioner in \eqref{first schur complement preconditioner} becomes more obvious in the following example. If the $d$-level approximating matrix sequences  $\{\{\tilde{D}_{\bm{n},m}\}_n\}_m,$ $\{\{\tilde{J}_{1_{\bm{n},m}}\}_n\}_m,$ and $\{\{\tilde{J}_{2_{\bm{n},m}}\}_n\}_m$ satisfy the conditions of Proposition \textnormal{\ref{Prop. ACS matrices}}, and belong to the set of $d$-level circulant matrices of size $N(\bm{n})\times N(\bm{n})$, i.e. $\mathcal{C}_{\bm{n}}$, then the preconditioner in \eqref{first schur complement preconditioner} will be cheap to form  and store, using the fast Fourier transform (requiring $O\big(N(\bm{n})\log(N(\bm{n}))\big)$ operations and $O\big(N(\bm{n})\big)$ memory). This is because $\mathcal{C}_{\bm{n}}$ is a commutative ring under standard matrix addition and multiplication (see Theorem \textnormal{\ref{theorem Circulant matrices}}).
\end{remark}
\subsubsection{A matching Schur complement approximation}
\par As mentioned earlier, many approximating sequences based on the GLT theory would not allow for an easy computation or storage of the preconditioner in \eqref{first schur complement preconditioner}. While the numerical results of this paper will not focus on this case, we present an alternative to the preconditioner in \eqref{first schur complement preconditioner}, which could allow one to use various GLT approximations for the blocks of the matrix in \eqref{generic ADMM subproblems 1and3 optimality conditions}, and form an easily computable Schur complement approximation for a matrix of the form of \eqref{generic ADMM normal equations}.
\par In what follows, we define a Schur complement approximation based on the matching strategy, which was proposed in \cite{paper_47} and has been applied in a wide range of applications (e.g. \cite{paper_2,paper_48,paper_65}). While this approach can be very general, it is particularly effective under some additional assumptions imposed on problem \eqref{discretized generic inverse problem}. More specifically, we study the properties of this approximation using the GLT theory, and give certain assumptions under which such an approach would be optimal.
\par Given three sequences  $\{\{\tilde{D}_{\bm{n},m}\}_n\}_m$, $\{\{\tilde{J}_{1_{\bm{n},m}}\}_n\}_m$, $\{\{\tilde{J}_{2_{\bm{n},m}}\}_n\}_m$, satisfying the conditions of Proposition \ref{Prop. ACS matrices}, we define the following matrix: 
\begin{equation*}
\hat{D}_{\bm{n},m} = \tilde{D}_{\bm{n},m}^* +  \rho^{\frac{1}{2}}\left(\tilde{J}_{1_{\bm{n},m}} + \frac{1}{\delta}I_{\bm{n}} \right)^{\frac{1}{2}}\left( \left(\rho\tilde{J}_{2_{\bm{n},m}} + \frac{\rho}{\delta}I_{\bm{n}}\right)^{-1} + \frac{\delta}{\rho} I_{\bm{n}}  \right)^{\frac{1}{2}},
\end{equation*}
\noindent using which we can define an approximation for the matrix in \eqref{generic ADMM normal equations} as
\begin{equation} \label{second schur complement preconditioner}
\hat{S}_{\bm{n},m} = \hat{D}_{\bm{n},m}\left( \left(\rho\tilde{J}_{2_{\bm{n},m}} + \frac{\rho}{\delta}I_{\bm{n}}\right)^{-1} + \frac{\delta}{\rho} I_{\bm{n}}  \right)^{-1}\hat{D}^*_{\bm{n},m}.
\end{equation}
\par For simplicity of exposition let us define the following matrices:
\[M_{1_{\bm{n}}} = \rho \left(J_{1_{\bm{n}}} + \frac{1}{\delta}I_{\bm{n}}\right),\ \ \tilde{M}_{1_{\bm{n}}} = \rho \left(\tilde{J}_{1_{\bm{n},m}} + \frac{1}{\delta}I_{\bm{n}}\right),\]
\[ M_{2_{\bm{n}}} = \left(\rho\left(J_{2_{\bm{n}}} + \frac{1}{\delta}I_{\bm{n}}\right)\right)^{-1} + \frac{\delta}{\rho} I_{\bm{n}},\ \ \tilde{M}_{2_{\bm{n}}} = \left(\rho\left(\tilde{J}_{2_{\bm{n},m}} + \frac{1}{\delta}I_{\bm{n}}\right)\right)^{-1} + \frac{\delta}{\rho} I_{\bm{n}} . \]
\noindent Under Assumption \ref{Assumption GLT matrices}, we have that $\{M_{1_{\bm{n}}}\}_n \sim_{GLT} \rho(\xi_1 + \delta^{-1})$, $\{M_{2_{\bm{n}}}\}_n  \sim_{GLT} (\rho(\xi_2 + \delta^{-1}))^{-1} + \frac{\delta}{\rho})$, and $\tilde{M}_{1_{\bm{n}}} \sim_{GLT} \rho(\xi_{1_m} + \delta^{-1})$ with $\xi_{1_m} \rightarrow \xi_1$ in measure, while $\tilde{M}_{2_{\bm{n}}} \sim_{GLT} (\rho(\xi_{2_m} + \delta^{-1}))^{-1} + \frac{\delta}{\rho})$, where $\xi_{2_m} \rightarrow \xi_2$ in measure. Further, notice that all four previous matrix-sequences are comprised of Hermitian and positive definite matrices, each of which admits a square root. 
\begin{lemma} \label{lemma square root of spd matrices}
Let $\bm{n} \in \mathbb{N}^d$ be a $d$-index and $\{A_{\bm{n}}\}_n$ be a multilevel matrix-sequence with  $A_{\bm{n}}$ being Hermitian positive definite of size $N(\bm{n})\times N(\bm{n})$ and $\{A_{\bm{n}}\}_n \sim_{GLT} \chi$, where $\chi$ is a measurable function $\chi:[0,1]^d \times[-\pi,\pi]^d \mapsto \mathbb{R}$ such that $\chi \geq 0$ and $\chi \neq 0$ a.e.. Then, the matrices $A_{\bm{n}}$ ($A^{-1}_{\bm{n}}$, respectively) admit a square root $A_{\bm{n}}^{\frac{1}{2}}$ ($A^{-\frac{1}{2}}_{\bm{n}}$, respectively), such that $\{A^{\frac{1}{2}}_{\bm{n}}\}_n \sim_{GLT} \chi^{\frac{1}{2}}$ ($\{A^{-\frac{1}{2}}_{\bm{n}}\}_n \sim_{GLT} \chi^{-\frac{1}{2}}$, respectively).
\end{lemma}
\begin{proof}
Let the function $f : (0,\infty) \mapsto (0,\infty)$, be defined as $f(x) = x^{\frac{1}{2}}$. Then, from Theorem \ref{theorem GLT properties} (condition (7.)), we have that $\{f(A_{\bm{n}})\}_n \sim_{GLT} f(\chi)$, where $f(A_{\bm{n}})$ is interpreted as a matrix function applied to the eigenvalues of matrix $A_{\bm{n}}$. 
\end{proof}
\begin{theorem} \label{theorem second schur complement preconditioner}
Let Assumption \textnormal{\ref{Assumption GLT matrices}} hold, and assume that we have available the sequences $\{\{\tilde{D}_{\bm{n},m}\}_n\}_m$, $\{\{\tilde{J}_{1_{\bm{n},m}}\}_n\}_m,$ and $\{\{\tilde{J}_{2_{\bm{n},m}}\}_n\}_m$, satisfying the conditions of Proposition \textnormal{\ref{Prop. ACS matrices}}. By defining $\hat{S}_{\bm{n},m}$ as in \eqref{second schur complement preconditioner}, we have:
\begin{itemize}
\item $\{\hat{S}_{\bm{n},m}\}_n \xrightarrow{\textnormal{a.c.s.}} \{S_{\bm{n}} + E_{\bm{n}}\}_n$, where 
\[\{E_{\bm{n}}\}_{n} \coloneqq \bigg\{M^{\frac{1}{2}}_{1_{\bm{n}}}M^{-\frac{1}{2}}_{2_{\bm{n}}} D_{\bm{n}} + D^*_{\bm{n}}  M^{-\frac{1}{2}}_{2_{\bm{n}}} M^{\frac{1}{2}}_{1_{\bm{n}}} \bigg\}_n \sim_{GLT} \epsilon,\]
\noindent with 
\[\epsilon \coloneqq  \rho(\xi_{1} + \delta^{-1})^{\frac{1}{2}}((\xi_{2} + \delta^{-1})^{-1} + \delta)^{-\frac{1}{2}} (\kappa + \bar{\kappa}). \]
\noindent Furthermore, $\{\hat{S}_{\bm{n},m}\}_n \sim_{GLT} \tau_m + \epsilon_m$ and $\tau_m + \epsilon_m \rightarrow \tau + \epsilon$ in measure, where $\tau$ is defined in Theorem \textnormal{\ref{theorem GLT preserving method}}, and $\tau_m,\ \epsilon_m$, $\epsilon$ are measurable functions having the same domain as $\tau$.  If $\tilde{E}_{\bm{n},m} \coloneqq \tilde{M}^{\frac{1}{2}}_{1_{\bm{n},m}}\tilde{M}^{-\frac{1}{2}}_{2_{\bm{n},m}} \tilde{D}_{\bm{n},m} + \tilde{D}^*_{\bm{n},m}  \tilde{M}^{-\frac{1}{2}}_{2_{\bm{n},m}} \tilde{M}^{\frac{1}{2}}_{1_{\bm{n},m}}$ is positive semi-definite for all $m$ and $n$, then the sequence of preconditioned normal equations' matrices is such that: 
\[\{\hat{S}_{\bm{n},m}^{-1} S_{\bm{n}}\}_n \sim_{GLT} \tau(\tau_m + \epsilon_m)^{-1} \rightarrow \tau (\tau + \epsilon)^{-1},\textnormal{ as } m \rightarrow \infty,\]
\noindent and there exist positive constants $C_{\dagger_1}$, $C_{\dagger_2}$, independent of $n,\ m$, such that $\lambda(\hat{S}_{\bm{n},m}^{-1} S_{\bm{n}}) \in [C_{\dagger_1},C_{\dagger_2}]$, for all $n,\ m$.
\item If the matrix sequences $\{ J_{1_{\bm{n}}}\}_n$, $\{J_{2_{\bm{n}}}\}_n$, and $\{D_{\bm{n}}\}_n$ are such that $J_{1_{\bm{n}}}$ and $J_{2_{\bm{n}}}$ are scaled identities or zero matrices, while $D_{\bm{n}} + D^*_{\bm{n}}$ is Hermitian positive semi-definite, then the matrix-sequence $\{\hat{S}_{\bm{n},m}^{-1} S_{\bm{n}}\}_n$ is weakly clustered at $[\frac{1}{2},1]$. If furthermore the matrix-sequence $\{\tilde{D}_{\bm{n},m}^{-1}D_{\bm{n}}\}_n$ is strongly clustered at $1$, then the matrix-sequence $\{\hat{S}_{\bm{n},m}^{-1} S_{\bm{n}}\}_n$ is strongly clustered at $[\frac{1}{2},1]$.
\end{itemize}
\end{theorem}
\begin{proof}
\par Firstly, notice that from \eqref{second schur complement preconditioner} we obtain the following expression:
\[\hat{S}_{\bm{n},m} = \tilde{S}_{\bm{n},m} + \tilde{M}^{\frac{1}{2}}_{1_{\bm{n}}}\tilde{M}^{-\frac{1}{2}}_{2_{\bm{n}}} \tilde{D}_{\bm{n}} + \tilde{D}^*_{\bm{n}}  \tilde{M}^{-\frac{1}{2}}_{2_{\bm{n}}} \tilde{M}^{\frac{1}{2}}_{1_{\bm{n}}},\]
\noindent where $\tilde{S}_{\bm{n},m}$ is defined as in \eqref{first schur complement preconditioner}. Then, the first part of the Theorem can be proved by employing Lemma \ref{lemma square root of spd matrices} and by performing a similar analysis to that of the proof of the first and third conditions of Theorem \ref{theorem first schur complement preconditioner}. For brevity, the latter is omitted.
\par We proceed by proving the second condition. Notice that if $J_{1_{\bm{n}}}$ and $J_{2_{\bm{n}}}$ are scaled identities or zero matrices (the latter being mostly of theoretical interest), then we can represent them exactly, that is $\tilde{J}_{1_{\bm{n},m}} = J_{1_{\bm{n}}}$ and $\tilde{J}_{2_{\bm{n},m}} = J_{2_{\bm{n}}}$, for all $m$ and $n$. The latter implies that $M_{1_{\bm{n}}}$, $M_{2_{\bm{n}}}$ are scaled identities and we can write $M_{\bm{n}} = M_{1_{\bm{n}}} = \frac{1}{c_s} M_{2_{\bm{n}}}$, for some positive constant $c_s$. We define the following matrix:
\[ \bar{S}_{\bm{n}} = \frac{1}{c_s}(D^*_{\bm{n}} + \sqrt{c_s}M_{\bm{n}})M_{\bm{n}}^{-1}(D_{\bm{n}} + \sqrt{c_s}M_{\bm{n}}).\]
\noindent Following exactly the developments in \cite[Theorem 4.1]{paper_65} (since $D + D^* \succeq  0$), we can consider the generalized eigenproblem $\bar{S}^{-1}_{\bm{n}}S_{\bm{n}} x = \mu x$, and show that $\lambda(\bar{S}^{-1}_{\bm{n}} S_{\bm{n}}) \in [\frac{1}{2},1]$, where $S_{\bm{n}}$ is defined as in \eqref{generic ADMM normal equations}, $\mu$ is an arbitrary eigenvalue of the preconditioned matrix $\bar{S}^{-1}_{\bm{n}}S_{\bm{n}}$ and $x$ the corresponding eigenvector. 
\par Let us now notice that by Assumption \ref{Assumption GLT matrices}, the matrix-sequence $\{\bar{S}_{\bm{n}}\}_n$ is a GLT sequence. In particular, it is easy to see that $\{S_{\bm{n}} + E_{\bm{n}}\}_n \equiv \{\bar{S}_{\bm{n}}\}_n$ and hence its symbol is $\tau + \epsilon$, where $\epsilon$ is defined in the first condition of this Theorem. Again, from the first condition of this Theorem, we have that the preconditioner defined in \eqref{second schur complement preconditioner} is such that $\{\hat{S}_{\bm{n},m}\}_n \xrightarrow{\textnormal{a.c.s.}} \{S_{\bm{n}} + E_{\bm{n}}\}_n \equiv \{\bar{S}_{\bm{n}}\}_n$, and $\{\hat{S}_{\bm{n},m}\}_n \sim_{GLT} \tau_m + \epsilon_m$ with $\tau_m + \epsilon_m \rightarrow \tau + \epsilon$ in measure. Then, from Theorem \ref{theorem a.c.s. algebra} we know that $\{\hat{S}^{-1}_{\bm{n},m}\}_n \xrightarrow{\textnormal{a.c.s.}} \{\bar{S}^{-1}_{\bm{n}}\}_n$. Using Definition \ref{definition a.c.s.}, we have that for all $n \geq n_m$, we can write:
\begin{equation} \label{eq: 1 thm matching prec }
\bar{S}^{-1}_{\bm{n}} = \hat{S}^{-1}_{\bm{n},m} + R_{\bm{n},m} + N_{\bm{n},m},\quad \textnormal{rank}(R_{\bm{n},m}) \leq c(m)N(\bm{n}),\quad \|N_{\bm{n},m}\| \leq \omega(m), 
\end{equation}
\noindent where $n_m$, $c(m)$ and $\omega(m)$ depend only on $m$ and are such that:
\begin{equation*} 
 \lim_{m \rightarrow \infty} c(m) = \lim_{m \rightarrow \infty} \omega(m) = 0.
 \end{equation*}
\noindent In view of the previous, we can analyze the sequence $\{\bar{S}^{-1}_{\bm{n}}S_{\bm{n}} - \hat{S}^{-1}_{\bm{n},m}S_{\bm{n}}\}_n$ as follows:
\begin{equation*}
\begin{split}
\bar{S}^{-1}_{\bm{n}}S_{\bm{n}} - \hat{S}^{-1}_{\bm{n},m}S_{\bm{n}} = (\bar{S}^{-1}_{\bm{n}} - \hat{S}^{-1}_{\bm{n},m})S_{\bm{n}} & = (R_{\bm{n},m} + N_{\bm{n},m})S_{\bm{n}},
\end{split}
\end{equation*}
\noindent where $\textnormal{rank}(R_{\bm{n},m}S_{\bm{n}}) \leq \textnormal{rank}(R_{\bm{n},m}) \leq c(m) N(\bm{n})$ and $\|N_{\bm{n},m}S_{\bm{n}}\|  \leq \omega(m) \|S_{\bm{n}}\| = \Theta(\omega(m))$. In other words, we have that $\{ \bar{S}^{-1}_{\bm{n}}S_{\bm{n}} - \hat{S}^{-1}_{\bm{n},m}S_{\bm{n}}\}_n$ is weakly clustered at zero. Furthermore, as $\lambda(\bar{S}^{-1}_{\bm{n}} S_{\bm{n}}) \in [\frac{1}{2},1]$, we conclude that  $\{\hat{S}^{-1}_{\bm{n},m}S_{\bm{n}}\}_n$ is weakly clustered at $[\frac{1}{2},1]$. 
\par Finally, if we assume that $\{\tilde{D}^{-1}_{\bm{n},m}D_{\bm{n}}\}_n$ is strongly clustered at $1$, and by noting that $M_{\bm{n}}$ is a scaled identity (and hence $\tilde{M}_{\bm{n},m} = M_{\bm{n}},$ for all $n,\ m$), we can conclude that \eqref{eq: 1 thm matching prec } holds for $R_{\bm{n},m}$ such that $\textnormal{rank}(R_{\bm{n},m}) = O(1)$. By employing a similar methodology as before, this yields that $\{\hat{S}^{-1}_{\bm{n},m}S_{\bm{n}}\}_n$ is strongly clustered at $[\frac{1}{2},1]$. 
\end{proof}

\begin{remark}
Let us now briefly discuss the applicability of the preconditioner in \eqref{second schur complement preconditioner}. Firstly, it is important to notice that such a preconditioner is generally only sensible when the approximating matrices $\tilde{J}_{1_{\bm{n},m}}$ and $\tilde{J}_{2_{\bm{n},m}}$ are diagonal, circulant or zero. If this is not the case, we discuss a remedy in Section \textnormal{\ref{subusubsection Generalized ADMM}}. In many applications of interest, the preconditioner $\tilde{D}_{\bm{n},m}$ has a diagonal times a multilevel banded Toeplitz structure (e.g. \textnormal{\cite{CD,paper_60}}). The application of the preconditioner in \eqref{second schur complement preconditioner} consists in a single (LU or, if applicable, Cholesky) factorization of $\hat{D}_{\bm{n},m}$ at the beginning of the optimization, and subsequently two backward solves for every ADMM iteration. Such an approach should be feasible, in terms of memory and computational requirements, as long as the problem dimensions are not very large and the bandwidth of the matrix $\tilde{D}_{\bm{n},m}$ is small. For high-dimensional problems, one could employ an incomplete or specialized factorization (e.g. \textnormal{\cite{CD}}), possibly assisted by suitable low-rank updates, if necessary. In some cases, replacing the factorization with a specialized iterative solver (such as a multigrid method as in \textnormal{\cite{paper_60}}) could be beneficial. However, it is important to note that factorization (complete or incomplete) needs to be computed only once. An alternative employing low-rank approximations of the associated matrices is discussed in the following Remark. The suitability of each of the aforementioned approaches depends heavily on the problem under consideration.
\end{remark}
\begin{remark}
As mentioned earlier, the proposed preconditioner in \eqref{second schur complement preconditioner} allows one to use a variety of approximations for the blocks of the matrix in \eqref{generic ADMM subproblems 1and3 optimality conditions}, based on the GLT theory, under certain conditions (which hold for a wide class of problems, such as the problem considered in Section \textnormal{\ref{section FDE problem})}. In fact, this preconditioner can be seen as an approximation of the preconditioner in \eqref{first schur complement preconditioner}, which in turn has limited applicability unless the approximating blocks have a multilevel circulant structure. 
\par The limitations of preconditioner \eqref{second schur complement preconditioner}  depend on the problem under consideration. In particular, if the assumptions of the second condition of Theorem \textnormal{\ref{theorem second schur complement preconditioner}} hold, then it can serve as a basis for constructing easily computable optimal preconditioners. 
Furthermore, if the aforementioned assumptions hold for the problem under consideration, one might be able to use a tensor product approach with low-rank approximations of the matrices in \eqref{generic ADMM subproblems 1and3 optimality conditions} to solve problem \eqref{discretized generic inverse problem}. Such solvers can be extremely effective, allowing one to solve high-dimensional problems, however they tend to require that various features of the problem (e.g. initial conditions, desired state, boundary conditions, discrete solutions) are approximated in a low-rank format, which is not always the case. The proposed approach would allow one to create a rather general tensor product solver for inverse problems measuring the discrepancy of the state variable $\rm{y}$ from a desired state $\rm{\bar{y}}$ as well as the size of the control $\rm{u}$ in the $L^2$-norm, where the structure of the problem allows this. Such solvers have been proposed in \textnormal{\cite{paper_2,HKKS}} for the equality constrained case, and hence the proposed methodology could allow one to further generalize these approaches. Additionally, many low-rank tensor product solvers in the literature require that the objective function has a scaled identity Hessian. This can be alleviated here, by making use of the generalized ADMM presented in Section \textnormal{\ref{subusubsection Generalized ADMM}}, alongside the preconditioner in \eqref{second schur complement preconditioner}.
\end{remark}
\subsubsection{Element-wise Schur complement approximation}

\par In the context of finite element methods, a popular preconditioner is the so-called element-wise (also known as additive or element-by-element) Schur complement approximation. As this approach has been analyzed multiple times, we only mention it here as a viable alternative for preconditioning the normal equations in  \eqref{generic ADMM normal equations} and refer the interested reader to the available literature. In particular, such preconditioners have been analyzed using the GLT theory in \cite{paper_49,paper_50}. An analysis for general problems can be found in \cite{paper_51} and the references therein. These preconditioners can be very effective (in fact optimal under reasonable and general assumptions). Furthermore, they can efficiently be implemented in a parallel environment, allowing one to solve huge-scale problem instances (see e.g. \cite{paper_52,paper_50}).

\subsection{General quadratic objective function}
\par As we stressed earlier, it could be the case that both $J_{1_{\bm{n}}}$ and $J_{2_{\bm{n}}}$ are general positive semi-definite matrices, whose inverses (if they exist) are expensive to compute. As a consequence, the normal equations could be prohibitively expensive to form. In order to tackle such problems, we propose two alternatives. The former simply avoids forming the normal equations and solves \eqref{generic ADMM subproblems 1and3 optimality conditions} instead, using an appropriate Krylov subspace method. The latter approach generalizes the algorithmic framework in Algorithm \ref{ADMM algorithm}, allowing us to simplify the resulting sub-problems. Then, the simplified sub-problems can be solved using PCG alongside any of the previously presented Schur complement approximations.
\subsubsection{A saddle point approximation}
\par In many applications, forming a Schur complement of system \eqref{generic ADMM subproblems 1and3 optimality conditions} would be very costly. Instead, one could solve system \eqref{generic ADMM subproblems 1and3 optimality conditions}, which can be seen as a regularized saddle point system. Among many other iterative methods, one could employ preconditioned MINRES to solve systems of this form. The aforementioned method allows only the use of a positive definite preconditioner, hence, many block preconditioners for \eqref{generic ADMM subproblems 1and3 optimality conditions} are not applicable. For instance, block-triangular preconditioners, motivated by the work in \cite{Ipsen01,MGW00}, would generally require a non-symmetric solver such as GMRES \cite{SaSc86}. However, block-diagonal preconditioners have been shown to be very effective and efficient in practice for systems of the form of \eqref{generic ADMM subproblems 1and3 optimality conditions} (see for example \cite{paper_43,paper_44,paper_42}). To that end, we can define the following positive definite block-diagonal preconditioner:

\begin{equation} \label{generic block-diagonal preconditioner}
\tilde{K}_{\bm{n},m} = \begin{bmatrix}
\rho(\tilde{J}_{1_{\bm{n},m}} + \frac{1}{\delta} I_{\bm{n}}) & 0 & 0\\
0 & \rho(\tilde{J}_{2_{\bm{n},m}} + \frac{1}{\delta} I_{\bm{n}}) & 0\\
0 & 0 & \tilde{S}_{\bm{n},m}
\end{bmatrix},
\end{equation}
\noindent where $\tilde{S}_{\bm{n},m}$ can be defined as in \eqref{first schur complement preconditioner} or as in \eqref{second schur complement preconditioner} (and indeed any other suitable Schur complement approximation), assuming that we have available two sparse sequences of $d$-level GLT sequences $\{\{\tilde{J}_{1_{\bm{n},m}}\}_n\}_m$ and $\{\{\tilde{J}_{2_{\bm{n},m}}\}_n\}_m$, satisfying the conditions of Proposition \ref{Prop. ACS matrices}. We note that preconditioners similar to \eqref{generic block-diagonal preconditioner} have been analyzed multiple times in the literature and hence such an analysis is omitted here (see for example \cite{paper_43,paper_44,paper_45,paper_46,paper_42}). It is important to notice that the quality of the preconditioner in \eqref{generic block-diagonal preconditioner} depends heavily on the quality of the Schur complement approximation, as well as on the approximations of the $(1,1)$ and $(2,2)$ blocks of the matrix in \eqref{generic ADMM subproblems 1and3 optimality conditions}, which can be computed by making use of the GLT theory.

\subsubsection{Generalized ADMM} \label{subusubsection Generalized ADMM} 

\par Instead of solving the saddle point system in \eqref{generic ADMM subproblems 1and3 optimality conditions}, one could derive the following generalized ADMM algorithm, as described in Algorithm \ref{generalized ADMM algorithm}. The following methodology is presented for completeness and is focused on the case where all the associated matrices as well as state and control variables are real. One could apply it to the complex case, however, in that case the theory derived in \cite{paper_53} to support such methods, would no longer hold.
\begin{algorithm}[!ht]
\caption{(2-Block) Generalized ADMM }
    \label{generalized ADMM algorithm}
    \textbf{Input:}  Let $y_{\bm{n}}^0,u_{\bm{n}}^0,z_{y_{\bm{n}}}^0,z_{u_{\bm{n}}}^0,p_{\bm{n}}^0,w_{y_{\bm{n}}}^0,w_{u_{\bm{n}}}^0 \in \mathbb{R}^{N(\bm{n})}$, $\delta > 0$, $\rho \in (0,1]$, $R_y \succ 0$, $R_u \succ 0$. 
\begin{algorithmic}
\For {($j = 0,1,\dotsc$)}
 \begin{subequations} 
\begin{align}
(y_{\bm{n}}^{j+1},u_{\bm{n}}^{j+1}) =&\  \underset{y_{\bm{n}},u_{\bm{n}}}{\arg\min}\bigg\{\mathcal{L}_{\delta}(y_{\bm{n}},u_{\bm{n}},z^{j}_{y_{\bm{n}}},z^{j}_{u_{\bm{n}}},p^j_{\bm{n}},w^j_{y_{\bm{n}}},w^j_{u_{\bm{n}}})\notag \\&  +\frac{1}{2}(y_{\bm{n}}-y_{\bm{n}}^j)^T R_{y}(y_{\bm{n}}-y_{\bm{n}}^j) + (u_{\bm{n}}-u_{\bm{n}}^j)^T R_{u}(u_{\bm{n}}-u_{\bm{n}}^j)\bigg\}  \label{generalized ADMM subproblem 1}\\ 
(z_{y_{\bm{n}}}^{j+1},z_{u_{\bm{n}}}^{j+1}) =&\  \underset{z_y \in [y_{a},y_{b}],\ z_u \in [u_{a},u_{b}]}{\arg\min} \big \{ \mathcal{L}_{\delta}(y_{\bm{n}}^{j+1},u_{\bm{n}}^{j+1},z_{y_{\bm{n}}},z_{u_{\bm{n}}},p_{\bm{n}}^j,w_{y_{\bm{n}}}^j,w_{u_{\bm{n}}}^j)\big\}   \label{generalized ADMM subproblem 2}\\ 
p_{\bm{n}}^{j+1} = &\ p_{\bm{n}}^j + \frac{\rho}{\delta}(D_{\bm{n}}y_{\bm{n}}^{j+1}+ u_{\bm{n}}^{j+1} - g_{\bm{n}})\label{generalized ADMM dual update 1}\\
(w_{y_{\bm{n}}}^{j+1},w_{u_{\bm{n}}}^{j+1}) =&\ \left(w_{y_{\bm{n}}}^j+ \frac{\rho}{\delta}(y_{\bm{n}}^{j+1}-z_{y_{\bm{n}}}^{j+1}), w_{u_{\bm{n}}}^j + \frac{\rho}{\delta}(u_{\bm{n}}^{j+1}-z_{u_{\bm{n}}}^{j+1})\right) \label{generalized ADMM dual update 2 and 3}
\end{align}
\end{subequations}
\EndFor
\end{algorithmic}
\end{algorithm}
\par There are two major differences between Algorithm \ref{ADMM algorithm} and Algorithm \ref{generalized ADMM algorithm}. In the latter method, we have added an extra proximal term in problem \eqref{ADMM subproblem 1}, which belongs to the class of Bregman distances, and indeed is produced by the Bregman function $\|\cdot\|_{R}$, where $R = R_y \oplus R_u$, $R_y \succ 0$, and $R_u \succ 0$. For a detailed derivation of proximal methods using Bregman distances, the reader is referred to \cite{paper_54} and the references therein. The second difference is that, in the general case, Algorithm \ref{generalized ADMM algorithm} requires that the step-size $\rho$ lies in a smaller interval than that allowed in Algorithm \ref{ADMM algorithm}. In fact, the allowed values for $\rho$ depend on the choice of $R_y$ and $R_u$. We refer the reader to \cite{paper_53} for a more general derivation of methods similar to Algorithm \ref{generalized ADMM algorithm}, in which a precise condition is given for the maximum allowed values of $\rho$, so that the method converges globally. Furthermore, the authors in \cite{paper_53} prove linear convergence of the method under different sets of conditions, one of which requires that $J_{1_{\bm{n}}} \succ 0$ and $J_{2_{\bm{n}}} \succ 0$. 
\par In light of the previous discussion, we can choose:
\[ R_{y_{\bm{n}}} = c_y I_{\bm{n}} - J_{1_{\bm{n}}},\quad  R_{u_{\bm{n}}} = c_u I_{\bm{n}} - J_{2_{\bm{n}}},\]
\noindent where $c_y,\ c_u > 0$ are such that $R_{y_{\bm{n}}} \succ 0$,  $R_{u_{\bm{n}}} \succ 0$. With these choices of $R_{y_{\bm{n}}}$ and $R_{u_{\bm{n}}}$, the optimality conditions of \eqref{generalized ADMM subproblem 1} and \eqref{generalized ADMM dual update 1} involve the coefficient matrix:
\begin{equation} \label{generic generalized ADMM subproblems 1and3 optimality operator}
\begin{bmatrix}
\rho (c_y + \frac{1}{\delta}) I_{\bm{n}} & 0 & D_{\bm{n}}^*\\
0 & \rho(c_u + \frac{1}{\delta})I_{\bm{n}} &  I_{\bm{n}}\\
D_{\bm{n}} &  I_{\bm{n}} & -\frac{\delta}{\rho} I_{\bm{n}}
\end{bmatrix}.
\end{equation}
\noindent As one can easily observe, the normal equations operator of \eqref{generic generalized ADMM subproblems 1and3 optimality operator} can be efficiently applied to a vector, and all the previously presented Schur complement approximations can be used within PCG to accelerate the solution of the new simplified sub-problems. Furthermore, notice that this way we can ensure that the $(1,1)$ and $(2,2)$ blocks of the matrix in \eqref{generic generalized ADMM subproblems 1and3 optimality operator} are scaled identities, and hence the preconditioner in \eqref{second schur complement preconditioner} can be particularly effective (see Theorem \ref{theorem second schur complement preconditioner}). 
\par We should note at this point that similar methodologies can be employed to enforce certain structure on the associated matrices of problem \eqref{Generic ADMM Full problem}. While this can be very effective in some cases, by making the ADMM sub-problems easy to solve, it should be used with caution. On the one hand $\rho$ is required, in general, to take values in the interval $(0,1]$. In practice, the larger the value of $\rho$, the faster the convergence of ADMM. More importantly, if the constants $c_y$ and $c_u$ are large, we essentially regularize the problem strongly (i.e. we force a large $\delta$, in the case of Algorithm \ref{ADMM algorithm}). This means that tuning $\delta$ in Algorithm \ref{generalized ADMM algorithm} will not allow us to accelerate the algorithm significantly (which is not the case for Algorithm \ref{ADMM algorithm}).

\section{The FDE-Constrained Optimization Model} \label{section FDE problem}
 \par In this section, we present the FDE-constrained optimization problem studied hereon and provide details as to the FDE discretization used. We then highlight some important properties of the resulting discretized matrices. 
\par We define the \emph{Caputo derivative} of a function ${\rm f}(t)$ defined on $t\in[t_0,t_1]$, of real order $\alpha$ such that $n-1<\alpha<n$ with $n\in\mathbb{N}$, as follows:
\begin{equation*}
\ \DC{\alpha}{t_0}{t}\hspace{0.15em}{\rm f}(t)=\frac{1}{\Gamma(n-\alpha)}\int_{t_0}^{t}\frac{{\rm d}^{n}{\rm f}(s)}{{\rm d}s^n}\frac{{\rm d}s}{(t-s)^{\alpha-n+1}},
\end{equation*}
assuming convergence of the above \cite{DFFL,MT1,Podlubny}. We also define the left-sided and right-sided \emph{Riemann--Liouville derivatives} of a function ${\rm f}(x)$ defined on $x\in[x_0,x_1]$, of real order $\beta$ such that $n-1<\beta<n$ with $n\in\mathbb{N}$, as
\begin{align*}
\ \DRL{\beta}{x_0}{x}\hspace{0.15em}{\rm f}(x)={}&\frac{1}{\Gamma(n-\beta)}\frac{{\rm d}^{n}}{{\rm d}x^n}\int_{x_0}^{x}\frac{{\rm f}(s)\hspace{0.15em}{\rm d}s}{(x-s)^{\beta-n+1}}, \\
\ \DRL{\beta}{x}{x_1}\hspace{0.15em}{\rm f}(x)={}&\frac{(-1)^n}{\Gamma(n-\beta)}\frac{{\rm d}^{n}}{{\rm d}x^n}\int_{x}^{x_1}\frac{{\rm f}(s)\hspace{0.15em}{\rm d}s}{(s-x)^{\beta-n+1}},
\end{align*}
respectively. From this, we define the \emph{symmetric Riesz derivative} as follows \cite{Podlubny,SKM}:
\begin{equation}
\ \label{SymRiesz} \DR{\beta}{x}\hspace{0.15em}{\rm f}(x)=\frac{-1}{2 \cos(\frac{\beta \pi}{2})}\bigg(\hspace{-0.15em}\DRL{\beta}{x_0}{x}\hspace{0.15em}{\rm f}(x)+\DRL{\beta}{x}{x_1}\hspace{0.15em}{\rm f}(x)\bigg).
\end{equation}
\par We highlight that Caputo derivatives are frequently used for discretization of FDEs in time, given initial conditions, with Riemann--Liouville derivatives correspondingly considered for spatial derivatives, given boundary conditions. We consider the minimization problem:
\begin{equation} \label{FDEOpt}
\begin{split}
\text{min}_{{\rm y},{\rm u}} & \ \mathrm{J}({\rm y}(\bm{x},t),{\rm u}(\bm{x},t)) \\
 \text{s.t.}\ \  &   \left(\hspace{-0.15em}\DC{\alpha}{0}{t}-\DR{\beta_1}{x_1}-\DR{\beta_2}{x_2}\right){\rm y}(\bm{x},t) + {\rm u}(\bm{x},t) = {\rm g}(\bm{x},t), \\ 
 &\  {\rm y}_{a}(\bm{x},t) \leq {\rm y}(\bm{x},t) \leq {\rm y}_{b}(\bm{x},t),\quad {\rm u}_{a}(\bm{x},t) \leq {\rm u}(\bm{x},t) \leq {\rm u}_{b}(\bm{x},t),\\ 
\end{split}
\end{equation}
where the fractional differential equation and additional algebraic constraints are given on the space-time domain $\Omega\times(0,T)$, where $\Omega \subset \mathbb{R}^2$ has boundary $\partial\Omega$, and the spatial coordinates are given by $\bm{x}=[x_1,x_2]^T$.  We impose the initial condition ${\rm y}(\bm{x},0) =0$ at $t=0$, and the Dirichlet condition ${\rm y}=0$ on $\partial\Omega\times(0,T)$. We assume that the orders of differentiation satisfy $0<\alpha<1$, $1<\beta_1<2$, $1<\beta_2<2$. 
\par The cost functional $J({\rm y},{\rm u})$ measures the misfit between the \emph{state variable} ${\rm y}$ and a given \emph{desired state} $\bar{\rm y}$ in some given norm, and also measures the `size' of the \emph{control variable} $\rm{u}$. In this paper we consider the cost functional $J({\rm y},{\rm u})$ corresponding to $L^2$-norms measuring both terms:
\begin{equation} \label{FDE objective functional}
\ \mathrm{J}({\rm y},{\rm u})=\frac{1}{2}\int_{0}^{T}\int_{\Omega}({\rm y}-\bar{\rm y})^2~{\rm d}x{\rm d}t+\frac{\gamma}{2}\int_{0}^{T}\int_{\Omega}{\rm u}^2~{\rm d}x{\rm d}t.
\end{equation}
\noindent Here $\gamma>0$ denotes a regularization parameter on the control variable. We note that other variants for $\mathrm{J}({\rm y},{\rm u})$ are possible, including measuring the state misfit and/or the control variable in other norms, as well as alternative weightings within the cost functionals. We also emphasize that it is perfectly reasonable to consider such problems involving FDEs in one or three spatial dimensions (or indeed higher dimensions), rather than in two dimensions as in \eqref{FDEOpt}, and the methodology in this paper could be readily tailored to such problems.
\par Upon discretization, we consider the non-shifted Gr\"{u}nwald--Letnikov formula \cite{paper_2,Podlubny,SKM,paper_63} to approximate the Caputo derivative in time:
\begin{equation}
\ \label{C_GL} \DC{\alpha}{t_0}{t}\hspace{0.15em}{\rm y}(t) = \frac{1}{h_t^{\alpha}}\sum_{k=0}^{n_{t}-1}g_{k}^{\alpha}\hspace{0.15em}{\rm y}(t-k h_t) + O(h_t),
\end{equation}
where $h_t$ is the step-size in time, and $g_{k}^{\alpha}=\frac{\Gamma(k-\alpha)}{\Gamma(-\alpha)\Gamma(k+1)}$ may be computed recursively via $g_{k}^{\alpha}=(1-\frac{\alpha+1}{k})g_{k-1}^{\alpha}$, $k=1,2,...,\nu$, with $g_{0}^{\alpha}=1$ and $\nu \in \mathbb{N}$. This leads to the Caputo derivative matrix for all grid points in the time variable:
\begin{equation} \label{Caputo matrix}
\ \mathscr{C}^{\alpha}_{n_t}=\frac{1}{h_t^{\alpha}}\left[\begin{array}{ccccc}
g_0^{\alpha} & 0 & \cdots & \cdots  & 0 \\
g_1^{\alpha} & g_0^{\alpha} & \ddots &  & \vdots \\
\vdots & \ddots & \ddots & \ddots  & \vdots \\
\vdots &  & g_1^{\alpha} & g_0^{\alpha} & 0 \\
g_{n_{t}-1}^{\alpha} & \cdots & \cdots & g_1^{\alpha} & g_0^{\alpha} \\
\end{array}\right].
\end{equation}
\noindent For the (left-sided) spatial derivative we use the $p$-shifted Gr\"{u}nwald--Letnikov formula \cite{paper_62,MST,MT2,Podlubny}, with shift parameter $p = 1$, to minimize the local truncation error:
\begin{equation}
\ \label{RL_GL} \DRL{\beta}{x_0}{x}\hspace{0.15em}{\rm y}(x)=\frac{1}{h_x^{\beta}}\sum_{k=0}^{n}g_{k}^{\beta}\hspace{0.15em}{\rm y}(x-(k-1)h_x) + O(h_x),
\end{equation}
where $h_x$ is the step-size in space, leading to the matrix
\begin{equation*}
\ \mathscr{L}^{\beta,l}_n=\frac{1}{h_x^{\beta}}\left[\begin{array}{cccccc}
g_1^{\beta} & g_0^{\beta} & 0 & \cdots  & 0 \\
g_2^{\beta} & g_1^{\beta} & g_0^{\beta} & \ddots  & \vdots \\
\vdots & \ddots & \ddots  & \ddots & 0 \\
\vdots &  & g_2^{\beta} & g_1^{\beta} & g_0^{\beta} \\
g_n^{\beta} & \cdots & \cdots  & g_2^{\beta} & g_1^{\beta} \\
\end{array}\right],
\end{equation*}
whereby using the formula \eqref{SymRiesz} leads to the following Riemann--Liouville derivative matrix for the symmetrized Riesz derivative:
\begin{equation} \label{Riesz matrix}
\ \mathscr{L}^{\beta}_n=\frac{-1}{2 \cos(\frac{\beta \pi}{2})}\bigg(\mathscr{L}^{\beta,l}_n+(\mathscr{L}^{\beta,l}_n)^T\bigg).
\end{equation}
\noindent Using all the previous definitions, we can write the discretized version of the FDE constraint within \eqref{FDEOpt} as
\begin{equation} \label{FDE constraint equation}
 D_{\bm{n}}y_{\bm{n}} + u_{\bm{n}} = g_{\bm{n}},
\end{equation}
\noindent where $y_{\bm{n}}$, $u_{\bm{n}}$, $g_{\bm{n}}$ represent the discretized variants of ${\rm y}$, ${\rm u}$, ${\rm g}$, $\bm{n} = [n_x,n_y,n_t]$ is a $3$-index containing the grid sizes along each dimension, and
\begin{equation} \label{FDE constraint matrix}
D_{\bm{n}} = \mathscr{C}^{\alpha}_{n_t} \otimes I_{n_{x_1} \cdot n_{x_2}} - I_{n_t} \otimes\big( \mathscr{L}^{\beta_1}_{n_{x_1}} \otimes I_{n_{x_2}} +  I_{n_{x_1}}\otimes \mathscr{L}^{\beta_2}_{n_{x_2}} \big). 
\end{equation}
\noindent For simplicity of exposition, in the rest of the paper we assume that $h_{x_1} = h_{x_2} = h_x$, where $h_{x_i}$ is the discretization step in the respective spatial direction, noting that the method readily generalizes to problems where this is not the case.
\par By using the trapezoidal rule we approximate the two terms in the objective functional \eqref{FDE objective functional}, by
\begin{equation} \label{discretized Objective function}
J_{\bm{n}} = J_{1_{\bm{n}}} = \frac{1}{\gamma}J_{2_{\bm{n}}} = \begin{bmatrix}
I_{(n_t-1)\cdot n_{x_1} \cdot n_{x_2}} & 0\\
0 & \frac{1}{2} I_{n_{x_1} \cdot n_{x_2}}
\end{bmatrix}, 
\end{equation}
\noindent which is applied to vectors arising from every time-step, apart from the initial time $t = 0$. Notice that matrix $J_{\bm{n}}$ is diagonal with only two different values on the diagonal, and hence can be almost exactly approximated by a scaled identity. 

\par We should mention that we assume constant diffusion coefficients in the FDE constraints. In turn, this yields that the discretized constraint matrix has a multilevel Toeplitz structure. As we discuss in the following section, such matrices can be approximated by circulant preconditioners, which in turn allow us to use the preconditioner in \eqref{first schur complement preconditioner} for accelerating the solution of the resulting ADMM sub-problems. In the presence of non-constant diffusion coefficients, the discretized constraint matrices would belong to the class of multilevel GLT sequences. In this case, circulant preconditioners would no longer be effective and we would have to approximate such matrices using diagonal times multilevel banded Toeplitz matrices (see for example \cite{paper_61,paper_60}). In light of the discussion in Section \ref{Section A structure preserving method}, we can observe that one could extend the results presented in this paper to the non-constant diffusion coefficient case, by making use of the preconditioner in \eqref{second schur complement preconditioner} (upon noting that the discretization of the functional in \eqref{FDE objective functional} yields a diagonal matrix). For brevity of presentation this is left to a future study.
\par In the following proposition, we summarize some well-known properties of the fractional binomial coefficients that arise above when constructing the matrices $\mathscr{C}_{\alpha}$ and $\mathscr{L}_{\beta}$ (see for example \cite[page 397]{book_1}, or \cite{MT1,paper_18}):
\begin{proposition} \label{gamma prop}
Let $0 < \alpha < 1$ and $1 < \beta < 2$, with $g_k^{\alpha},\ g_k^{\beta}$ as in \eqref{C_GL}, \eqref{RL_GL}. Then, we have that:
\begin{equation}
\begin{split}
g_0^{\alpha} > 0,\ g_k^{\alpha} < 0,\ \forall\  k \geq 1,\quad \sum_{k=0}^{n_t} g_k^{\alpha} > 0,\ \forall\ n_t \geq 1, 
\end{split}
\end{equation}
\begin{equation}
\begin{split}
g_0^{\beta} = 1,\ g_1^{\beta} = -\beta,\ g_2^{\beta} > g_3^{\beta} > \dotsc > 0,\quad \sum_{k=0}^{\infty} g_k^{\beta} = 0,\ \sum_{k=0}^n g_k^{\beta} < 0,\ \forall\ n \geq  1.
\end{split}
\end{equation}
\end{proposition}

\section{Toeplitz Matrices and Circulant Preconditioners} \label{Section: Toeplitz and Circulant Matrices} 
\par In this section, we propose a multilevel circulant preconditioner, suitable for approximating multilevel Toeplitz matrices, and then examine the quality of such a preconditioner for the problem at hand, showing that the preconditioner is in fact a.c.s. for a scaled version of the coefficient matrix in \eqref{FDE constraint matrix}. 
\par Toeplitz and multilevel Toeplitz matrices appear often when (numerically) solving partial, integral, or fractional differential equations, problems in time series analysis, as well as in signal processing (see for example \cite{book_3,paper_6,paper_28,paper_27}, and the references therein).  An active area of research is that of solving a huge-scale systems of linear equations, $Ax = b$, where the matrix $A$ has some specific structure, such as Toeplitz, multilevel Toeplitz, or it can be written as a combination of Toeplitz and other structured matrices. There are two major approaches for solving such systems. One alternative is to solve them directly by exploiting the matrix structure (see for example \cite{paper_4,paper_56,paper_5,paper_55}). A more popular approach is to employ some iterative method to solve the system, assisted by an appropriately designed preconditioner, to ensure that the iterative method achieves fast convergence (as in \cite{paper_12,book_7,paper_9,paper_11,paper_15,paper_10,paper_3,paper_1,paper_14,paper_8,paper_7,paper_6,book_6,paper_13}).
An equally rich literature exists for preconditioning Toeplitz-like linear systems arising specifically from the discretization of fractional diffusion equations (see \cite{CD,paper_61,paper_58,DC,paper_57,paper_7,paper_6,paper_59,paper_60}, among others).
\par In this paper, we follow the simplest possible approach: that of approximating multilevel Toeplitz matrices using multilevel circulant preconditioners. To do so, we first have to derive a unilevel circulant approximation of an arbitrary unilevel Toeplitz matrix. Given a unilevel Toeplitz matrix $T_n \in \mathbb{R}^{n \times n}$, we employ the circulant approximation proposed for the first time in \cite{paper_3} (also called the T. Chan preconditioner for $T_n$). More specifically, we define the optimal circulant approximation of $T_n$, as the solution of the following optimization problem:
\begin{equation} \label{level 1 circulant optimization problem}
 C_1(T_n) = \min_{C_n \in \mathcal{C}_n} \| C_n-T_n\|_F, 
 \end{equation}
\noindent where $\mathcal{C}_n$ is the set of all $n \times n$ circulant matrices, and $\|\cdot\|_F$ the \emph{Frobenius norm}. It turns out that \eqref{level 1 circulant optimization problem} admits the following closed form solution: 
\[c_i = \frac{(n-i)\cdot t_i + i \cdot t_{-n + i}}{n},\ \ i \in \{0,\dotsc,n-1\}. \] 
\noindent Then, we can write $C_1(T_n) = F_n^* \Lambda_n F_n$, where $F_n$ is the discrete Fourier transform of size $n$ and $\Lambda_n$ is a diagonal matrix containing the eigenvalues of $C_1(T_n)$, which can be computed as $\Lambda_n = \textnormal{diag}(F_n c_1)$, where $c_1$ is the first column of $C_1(T_n)$.
\noindent Other unilevel circulant approximations are possible, such as those proposed in \cite{paper_15,paper_10, paper_13}, however, the T. Chan preconditioner seems (empirically) to behave better for the problem under consideration. \par We now focus on the discretized FDE given in (\ref{FDE constraint equation}). By multiplying this equation on both sides by $\psi = \min\{h_t^{\alpha},h_x^{\beta_1},h_x^{\beta_2}\}$, we have:
\[B_{\bm{n}}y_{\bm{n}} + \psi u_{\bm{n}} = \psi g_{\bm{n}},\]
\noindent where $y_{\bm{n}}$, $u_{\bm{n}}$, $g_{\bm{n}}$ represent the discretized variants of ${\rm y}$, ${\rm u}$, ${\rm g}$, $B_{\bm{n}} = \psi D_{\bm{n}}$, with $D_{\bm{n}}$ defined as in (\ref{FDE constraint matrix}), $h_t,\ h_x$ the time and spatial mesh-sizes, and $\bm{n} = [n_{x_1},n_{x_2},n_t]$. We observe that the matrix $D_{\bm{n}}$ (and hence $B_{\bm{n}}$) enjoys a 3-level Toeplitz structure. In particular, each block of $D_{\bm{n}}$ ($B_{\bm{n}}$) enjoys a quadrantally symmetric block Toeplitz structure (such matrices are analyzed for example in \cite{paper_9}). Given the matrix $B_{\bm{n}}$, we can define its T. Chan-based $3$-level circulant preconditioner as:
\begin{equation} \label{multilevel circulant preconditioner}
\begin{split}
C_3(B_{\bm{n}}) = &\ \psi C_1(\mathscr{C}^{\alpha}_{n_t}) \otimes I_{n_{x_1} \cdot n_{x_2}} - \psi I_{n_t} \otimes\big( C_1(\mathscr{L}^{\beta_1}_{n_{x_1}}) \otimes I_{n_{x_2}} +  I_{n_{x_1}}\otimes C_1(\mathscr{L}^{\beta_2}_{n_{x_2}}) \big) \\
=&\  (F_{n_{x_1}} \otimes F_{n_{x_2}} \otimes F_{n_t})^* \Lambda_{\bm{n}} (F_{n_{x_1}} \otimes F_{n_{x_2}} \otimes F_{n_t}),
\end{split}
\end{equation}
\noindent where $\Lambda_{\bm{n}}$ is the diagonal eigenvalue matrix of the preconditioner, computed as:
\[\Lambda_{\bm{n}} =\psi \Lambda_{\alpha} \otimes I_{n_{x_1}\cdot n_{x_2}} - \psi I_{n_t}\otimes \big(\Lambda_{\beta_1} \otimes I_{n_{x_2}} + I_{n_{x_1}} \otimes \Lambda_{\beta_2} \big),\]
\noindent with $\Lambda_{\alpha}$, $\Lambda_{\beta_1}$, $\Lambda_{\beta_2}$ being the diagonal matrices containing the eigenvalues of the T. Chan approximations of the matrices $\mathscr{C}^{\alpha}_{n_t}$, $\mathscr{L}^{\beta_1}_{n_{x_1}}$, and $\mathscr{L}^{\beta_2}_{n_{x_2}}$, respectively.
\par The preconditioner in \eqref{multilevel circulant preconditioner} can be computed efficiently in $O(N(\bm{n}) \log N(\bm{n}))$ operations, using the fast Fourier transform. The storage requirements are $O(N(\bm{n}))$ since we only need to store the eigenvalue matrix, that is $\Lambda_{\bm{n}}$. Clearly, the preconditioner in \eqref{multilevel circulant preconditioner} can be defined similarly for FDEs of arbitrary dimension, say $d$. Given a $d$-index $\bm{n}$, containing the level sizes of an arbitrary $d$-level Toeplitz $T_{\bm{n}}$ or circulant matrix $C_{\bm{n}}$, we summarize the computational and storage costs of various recursive linear algebra operations in Table \ref{Complexity table}.

\begin{table}[!ht]
\centering
\caption{Summary of computational and storage complexity\label{Complexity table}}
 \noindent   \begin{tabular}{rrrr}   
   \toprule  {\textbf{Structure}} &  {\textbf{Operation}} & {\textbf{Computations}} & {\textbf{Storage}}  \\ \midrule
    $d$-level circulant & $C_{\bm{n}} x$ & $O(N(\bm{n})\log N(\bm{n}))$ & $O(N(\bm{n}))$ \\ 
    $d$-level circulant & $C_{\bm{n}}^{-1}x$ & $O(N(\bm{n}) \log N(\bm{n}))$ & $O(N(\bm{n}))$ \\  
    $d$-level circulant & $C_{\bm{n}}^{(1)} C_{\bm{n}}^{(2)}$ & $O(N(\bm{n}))$ & $O(N(\bm{n}))$ \\ 
    $d$-level circulant & $C_{\bm{n}}^{(1)} + C_{\bm{n}}^{(2)} $ & $O(N(\bm{n}))$ & $O(N(\bm{n}))$ \\  
    $d$-level Toeplitz  & $T_{\bm{n}}x$ & $O(2^d N(\bm{n})\log N(\bm{n}))$& $O(2^d N(\bm{n}))$ \\ 
    $d$-level circulant & Construct $C_d(T_{\bm{n}})$ & $O(N(\bm{n})\log N(\bm{n}))$ & $O(N(\bm{n}))$ \\ \bottomrule
    \end{tabular}
\end{table}

\par Using the definition of the matrices used to construct matrix $B_{\bm{n}}$ (see \eqref{Caputo matrix} and \eqref{Riesz matrix}), we are now able to derive the generating function of this $3$-level Toeplitz matrix. To that end, let us define the following scalars: 
\begin{equation} \label{xi definition}
 \nu_1 = \frac{\psi}{h_x^{\beta_1}},\quad \nu_2 = \frac{\psi}{h_x^{\beta_2}},\quad \nu_3 = \frac{\psi}{h_t^{\alpha}},
\end{equation}
\noindent which are obviously bounded above by $1$, from the definition of $\psi$. Of course in order for these to be theoretically meaningful, we have to assume that $h_t^{\alpha} \propto h_x^{\beta_{1}} \propto h_x^{\beta_2}$.
\begin{lemma} \label{Lemma generating function of matrix B}
Let $\bm{n} = [n_{x_1}, n_{x_2},n_t]$ be a $3$-index and define the matrix $D_{\bm{n}}$ as in \eqref{FDE constraint matrix}. Then, the symbol generating the matrix-sequence $\{B_{\bm{n}}\}_n = \{\psi D_{\bm{n}} \}_n$, can be expressed as:
\begin{equation} \label{generating function of constraint matrix}
\begin{split}
 \phi_{\beta_1,\beta_2,\alpha}(\bm{\theta}) =  \nu_3 \sum_{k = 0}^{\infty} g_k^{\alpha} e^{i k \theta_3} - \qquad \qquad \qquad \qquad \qquad \qquad \\ \sum_{k = -1}^{\infty}\bigg(\frac{-\nu_1}{2 \cos(\frac{\beta_1 \pi}{2})}\big( g_{k+1}^{\beta_1} (e^{i k \theta_1} +  e^{-i k \theta_1})\big) + \frac{-\nu_2}{2 \cos(\frac{\beta_2 \pi}{2})}\big( g_{k+1}^{\beta_2}( e^{i k \theta_2} +  e^{-i k \theta_2})\big) \bigg),
\end{split}
\end{equation}
\noindent where $\bm{\theta} = [\theta_1,\theta_2,\theta_3]$, and $g_k^{c}$ is the fractional binomial coefficient, for some $c \in (0,1)\cup(1,2)$ and an arbitrary $k \geq 0$. Thus, we can write $B_{\bm{n}}  = T_{\bm{n}}(\phi_{\beta_1,\beta_2,\alpha})$.
\end{lemma}
\begin{proof}
\par We omit the proof, which follows easily from the definition of the matrices within $B_{\bm{n}}$, that is using the definition of $\mathscr{C}^{\alpha}_n$ in \eqref{Caputo matrix} and $\mathscr{L}^{\beta}_n$ in \eqref{Riesz matrix}. The reader is referred to \cite{paper_61,paper_6,paper_60}, among others, for derivations of similar results. The alternative representation of matrix $B_{\bm{n}}$ follows directly from Theorem \ref{theorem Toeplitz matrices}.
\end{proof}
\par To analyze the effectiveness of the proposed $3$-level circulant preconditioner for $B_{\bm{n}}$, we prove that the trigonometric polynomial generating function of matrix $B_{\bm{n}}$ belongs to the Wiener class (that is, it has absolutely summable coefficients). 
\begin{lemma} \label{lemma absolutely summable}
Assume that $\beta_1$ and $\beta_2$ are bounded away from $1$. Then, the generating function $\phi_{\beta_1,\beta_2,\alpha}(\bm{\theta})$ defined in \eqref{generating function of constraint matrix} belongs to the Wiener class, that is:
\[ \phi_{\beta_1,\beta_2,\alpha}(\bm{\theta}) = \sum_{\bm{k} \in \mathbb{Z}^3} \phi_{\bm{k}} e^{i \langle\bm{k}, \bm{\theta} \rangle},\ \textnormal{such that}\ \sum_{\bm{k} \in \mathbb{Z}^3}|\phi_{\bm{k}}| < \infty.\]
\end{lemma}
\begin{proof}
\par For brevity of presentation, we provide an outline of the proof. Firstly, one has to transform \eqref{generating function of constraint matrix} to the form $\phi_{\beta_1,\beta_2,\alpha}(\bm{\theta}) = \sum_{\bm{k} \in \mathbb{Z}^3} \phi_{\bm{k}} e^{i \langle\bm{k}, \bm{\theta} \rangle}$, by matching the coefficients of the associated trigonometric polynomials. By taking the  absolute values of the matched coefficients, applying the triangle inequality, and using the properties of the fractional binomial coefficients, summarized in Proposition \ref{gamma prop}, we obtain:
\begin{equation*}
\begin{split}
\sum_{\bm{k} \in \mathbb{Z}^3}|\phi_{\bm{k}}|  \leq\ & \nu_3 \sum_{k = 0}^{\infty}|g_k^{\alpha}| + \sum_{k = -1}^{\infty} \bigg(\frac{\nu_1}{|\cos(\frac{\beta_1 \pi}{2})|} |g_{k+1}^{\beta_1}| + \frac{\nu_2}{|\cos(\frac{\beta_2 \pi}{2})|}|g_{k+1}^{\beta_2}|\bigg)\\  \leq \ &(2 \nu_3)\cdot g_0^{\alpha} + \bigg(\frac{2  \nu_1}{|\cos(\frac{\beta_1 \pi}{2})|}\bigg) \beta_1 + \bigg(\frac{2 \nu_2}{|\cos(\frac{\beta_2 \pi}{2})|}\bigg) \beta_2.
\end{split}
\end{equation*}
\noindent The latter completes the proof.
\end{proof}
\noindent Using the results presented in \cite{paper_9,paper_8,paper_7}, we can derive the following Theorem, which in fact shows that the $3$-level circulant approximation of matrix $B_{\bm{n}}$ defined in \eqref{multilevel circulant preconditioner} is an a.c.s. for it.
\begin{theorem} \label{Theorem preconditioner error}
\noindent Let $B_{\bm{n}} = \psi D_{\bm{n}}$ where $\bm{n} = [n_{x_1},n_{x_2},n_t]$, and $C_3(B_{\bm{n}})$ its circulant approximation defined in \eqref{multilevel circulant preconditioner}. For every $\epsilon(m) > 0$, such that $\epsilon(m) \rightarrow 0$ as $m \rightarrow \infty$, there exist constants $N_{x_1},N_{x_2},N_t$, such that for all $n_{x_1} > N_{x_1}$, $n_{x_2} > N_{x_2}$, $n_t > N_t$:
\[B_{\bm{n}} - C_3(B_{\bm{n}}) = U_{\bm{n},\epsilon(m)} + V_{\bm{n},\epsilon(m)},\]
\noindent where
\[\textnormal{rank}(U_{\bm{n},\epsilon(m)}) = O(n_{x_2}  n_{x_1} + n_{x_1} n_t + n_t n_{x_2}),\quad \|V_{\bm{n},\epsilon(m)}\|_2 < \epsilon(m).\]
\end{theorem}
\begin{proof}
\noindent The proof is omitted since it follows exactly the developments in \cite[Theorem 3.2 and Theorem 4.1]{paper_7}, with the only difference being that the Strang unilevel circulant approximation is used there (for example see \cite{paper_15}) instead of the T. Chan approximation. Notice that the authors in \cite{paper_7}  assume invertibility  of $C_3(B_{\bm{n}})$, using which they prove a weak clustering result. Hence, to prove the result stated here, one needs to follow only part of the proof outlined in \cite[Theorem 3.2]{paper_7}. 
\end{proof}

\begin{remark} Following \textnormal{\cite[Remark 4.1]{paper_7}}, assuming that $d = O(1)$, we can recursively extend the result of Theorem \textnormal{\ref{Theorem preconditioner error}} to the $d$-level case, using induction. In other words, the developments discussed in this paper can be extended trivially to higher dimensional FDEs. As expected, the circulant approximation becomes weaker as the dimension of the associated FDE is increased. In particular, the result in \textnormal{\cite{paper_38}} shows that in the general case, any multilevel circulant preconditioner for multilevel Toeplitz matrices is not a superlinear preconditioner. Superlinear preconditioners are important, in that they allow preconditioned Conjugate Gradient-like methods to converge in a constant number of iterations, independently of the size of the problem. In general, one could not hope of achieving a strong clustering when preconditioning multilevel Toeplitz matrices using multilevel circulant preconditioners. In light of that, it comes as no surprise that a preconditioner like the one in \textnormal{\eqref{multilevel circulant preconditioner}} does not asymptotically capture all of the eigenvalues of the approximated multilevel Toeplitz matrix.
\end{remark}
\begin{remark}
Let us now notice that a scaled identity approximation for the discretized objective Hessian matrix in \textnormal{\eqref{discretized Objective function}} yields (trivially) a GLT sequence. Similarly, the approximation $C_3(B_{\bm{n}})$ in \textnormal{\eqref{multilevel circulant preconditioner}} for the matrix $B_{\bm{n}} = \psi D_{\bm{n}}$, where $D_{\bm{n}}$ is defined in \textnormal{\eqref{FDE constraint matrix}}, is also a GLT sequence (since it can be considered as a multilevel Toeplitz matrix). In view of the previous, as well as Theorem \textnormal{\ref{theorem GLT properties}} (condition (6.)), we can see that the proposed approximations for the matrices associated to the discretized version of \textnormal{\eqref{FDEOpt}} satisfy the conditions of Proposition \textnormal{\ref{Prop. ACS matrices}}. Hence, we are able to invoke Theorem  \textnormal{\ref{theorem first schur complement preconditioner}} for the preconditioner in \textnormal{\eqref{first schur complement preconditioner}}, which is constructed by using the aforementioned multilevel circulant approximations. Thus, we are able to show that the resulting preconditioned ADMM system matrix, corresponding to the normal equations in \eqref{generic ADMM normal equations}, is weakly clustered at $1$. Furthermore, by the same Theorem, we expect convergence of PCG in a number of iterations independent of the grid-size.
\end{remark}

\section{Implementation Details and Numerical Results} \label{Implementation and Numerical Results}
\par In this section we discuss specific implementation details and present the numerical results obtained by running the implementation of the proposed method over a variety of settings of the FDE optimization problem.
\subsection{Test problem and implementation details}
\par We assess the performance of the proposed method on the following test problem. We attempt to numerically solve problem \eqref{FDEOpt}. The state and the control are defined on the domain $\Omega \times (0,T) = (0,1)^2 \times (0,1)$. For some $n \in \mathbb{N}$, the discretized grid contains $n\times n \times  n $ uniform points, in space and time (i.e. we make use of the $3$-index $\bm{n} = [n, n,n]$), which yields:
\[x_1^i = ih_x,\ x_2^j = jh_x,\ t^k = kh_t,\ i,j = 1,\ldots, n,\ k = 1,\ldots,n,\ h_x = \frac{1}{n + 1},\ h_t = h_x.\] 
\noindent It is worth mentioning that the choice of the number of discretization points in time should depend on the value of the fractional derivative orders. In particular, in the theory we had to assume that $h_t^{\alpha} \propto h_x^{\beta_1} \propto h_x^{\beta_2}$. Of course, this could be difficult to satisfy for certain values of $\alpha,\ \beta_1$, and $\beta_2$. In terms of discretization error, $n_t = n$ suffices, as we employ first-order numerical schemes for the space and time fractional derivatives. In what follows, we choose to use $n_t = n$ throughout all the experiments, noting that for very large values of $n$, this should be adjusted to take into consideration the values of the fractional derivative orders. Such an increase in the number of discretization points in time, could potentially be tackled by the use of higher-order numerical methods for the space fractional derivatives (see \cite{paper_64} and the references therein for higher-order approximations for the Riemann--Liouville and Riesz fractional derivatives). This is omitted for a future study.
\par As a desired state function, we set $\bar{{\rm y}}(x_1,x_2,t) = 10 \cos(10x_1)\sin(x_1 x_2)(1-e^{-5t}),$  as in \cite[Section 5.1]{paper_2}, with homogeneous boundary and initial conditions. Throughout this section, we employ the convention that $n_{x_1} = n_{x_2}= n_t$, and we only present the overall size of the discretized state vector, that is $N(\bm{n}) = n_{x_1} \cdot n_{x_2} \cdot n_t = n_{x_1}^3$. As an indicator of convergence of the numerical method, we apply the trapezoidal rule to roughly approximate the discrepancy between the solution for the state and the desired state on the discrete level, i.e.:
\[ \mathcal{E}_{L^2}(y-\bar{y}) \approx \|\mathrm{y} - \bar{\mathrm{y}}\|_{L^2}.\]
\noindent We should note that the previous measure approximates the misfit between the state and the desired state of the continuous problem, and hence it is not expected to converge to zero. Due to the Dirichlet boundary conditions, there is a mismatch between $y$ and $\bar{y}$ on the boundary. Hence a refinement in the grid size is expected to result in slight increase in the approximate discrepancy measure.
\par We implement a standard 2-Block ADMM for solving problems of the form of \eqref{Generic ADMM Full problem}. The implementation follows exactly the developments in Section \ref{Section A structure preserving method}. We solve system \eqref{generic ADMM normal equations} using the MATLAB function \texttt{pcg}. We note that while various potential acceleration strategies for ADMMs have been studied in the literature (see for example \cite{paper_26,paper_29}), the focus of the paper is to illustrate the viability of the proposed approach, and hence the simplest possible ADMM scheme is adopted. The step-size of ADMM is chosen to be close to the maximum allowed one in all computations, that is $\rho = 1.618$. The termination criteria of the ADMM are summarized as follows:
\[\big(\|B_{\bm{n}}y^j_{\bm{n}} + \psi(u^j_{\bm{n}}-g_{\bm{n}})\|_{\infty} \leq 10^{-4} \big) \wedge \big(\|y^j_{\bm{n}}-z^j_{y_{\bm{n}}}\|_{\infty} \leq 10^{-4}\big) \wedge \big(\|u^j_{\bm{n}}-z^j_{u_{\bm{n}}}\|_{\infty} \leq 10^{-4}\big).\]
\noindent In order to avoid unnecessary computations, we do not require a specific tolerance for the dual infeasibility. Instead, we report the dual infeasibility at the accepted optimal point. The Krylov solver tolerance is set dynamically, based on the accuracy attained at the respective ADMM iteration. In particular, the required tolerance for the Krylov solver is set to: 
\[\textnormal{In. Tol.} = 0.05 \cdot \max\big\{\min\{\|B_{\bm{n}}y^j_{\bm{n}} + \psi(u^j_{\bm{n}}-g_{\bm{n}})\|_{\infty},\|y^j_{\bm{n}}-z^j_{y_{\bm{n}}}\|_{\infty},\|u^j_{\bm{n}}-z^j_{u_{\bm{n}}}\|_{\infty}\},10^{-4}\big\},\]
\noindent at every iteration $j$. Hence, we present the average number of inner iterations in the results to follow. Furthermore, we employ the convention that the discretized restricting functions are of the form $y_{b_{\bm{n}}} = -y_{a_{\bm{n}}} = c \cdot  e_{\bm{n}}$ (or $u_{b_{\bm{n}}} = -u_{a_{\bm{n}}} = c \cdot e_{\bm{n}}$), where $e_{\bm{n}}$ is the $N(\bm{n})$-dimensional vector of ones and $c > 0$. Thus, we present only the value of the entries of $y_{a_{\bm{n}}}$ ($u_{a_{\bm{n}}}$, respectively). 
\par As we discussed earlier, the FDE constraints were scaled by the constant $\psi$, since this was required from the theory (see Theorem \ref{Theorem preconditioner error}). By doing this, we ensure that the elements of the matrix $B_{\bm{n}}$ are of order 1 (assuming that $h_t^{\alpha} \propto h_{x_i}^{\beta_i},$ for $i = 1,2$). As a result, the discretized control in the FDE constraints is multiplied by $\psi$. In ADMM such a scaling translates to a scaled step of the dual variables corresponding to the FDE constraints. In order to improve the balance of the algorithm, we multiply by $\psi$ the constraints linking $u_{\bm{n}}$ with its copy variables $z_{u_{\bm{n}}}$, thus scaling all the dual multipliers corresponding to these constraints.
\par The penalty parameter of ADMM, $\delta$, is chosen from a pool of five values which deliver reasonably good behavior of the method. More specifically, for the experiments to follow we choose $\delta \in \{0.1,0.4,2,10,100\}$. We note here that one could tune this parameter for each problem instance and obtain significantly better results. However, as this is not practical, we restrict ourselves to a small set of possible values. 
\par The experiments were conducted on a PC with a 2.2 GHz Intel (hexa-) core i7 processor, run under the Windows 10 operating system. The code is written in MATLAB R2019a. 

\subsection{Numerical Results}
\par We distinguish three types of problems:
\begin{itemize}
\item Problems with box constraints on the state $\rm{y}$,
\item problems with box constraints on the control $\rm{u}$, and 
\item problems with box constraints on both variables.
\end{itemize}
\noindent As expected and verified in practice, the third type of problem is the most difficult one. Hence, we will focus our attention on problems with box constraints on both variables, while presenting a few experiments on problems of the other two types.
\paragraph{\textbf{Box constraints on the state} $\bm{{\rm y}}$}
\noindent Let us briefly focus on the case where the state variable is required to stay in a box, while the control is free, that is ${\rm y_a} \leq {\rm y} \leq {\rm y_b},\ -\infty \leq {\rm u} \leq \infty$. Using similar arguments as in \cite{paper_2,DC}, we can see that an optimal solution in this case is guaranteed to exist. We run the method for different inequality bounds on the state ${\rm y}$. The results are summarized in Table \ref{Experiment Full: Constraints on state}. All fixed parameters are provided at the title of the respective Table.
\begin{table}[!ht]
\centering
\caption{Inequalities on the state: Varying restriction bounds (with $N = 50^3$, $\beta_x = \beta_y = 1.3$,  $\alpha = 0.7$, $\gamma = 10^{-4}$, $\delta = 0.1$).\label{Experiment Full: Constraints on state}}
\scalebox{0.8}{
\begin{tabular}{rrrrrrr}
     \toprule
    \multirow{2}{*}{$\bm{y_a}$}      & \multirow{2}{*}{$\bm{\mathcal{E}_{L^2}(y-\bar{y})}$} & \multirow{2}{*}{\textbf{Dual Inf.}}   & \multicolumn{2}{c}{\textbf{Iterations}}      & \multirow{2}{*}{\textbf{Time (s)}}  \\   \cline{4-5}
 & & & {PCG} & {ADMM} & \\ \midrule

  $-7$ & 5.60 $\times 10^{-1}$(*)\tablefootnote{(*) means that the solution coincides with the equality constrained solution; all the variables lie strictly within the restriction bounds.} &  2.13 $\times 10^{-3}$  & 9 & 75 & 142.31 \\
      $-5$ &5.80 $\times 10^{-1}$&  3.14 $\times 10^{-3}$ & 10 & 105 & 206.73 \\
    $-3$ & 7.88 $\times 10^{-1}$&  4.11 $\times 10^{-3}$ & 10 & 100 & 193.77 \\
  $-1$ & 1.38 $\times 10^{0}$&  8.74 $\times 10^{-4}$ & 10 & 86 & 173.49 \\ \bottomrule
\end{tabular}}
\end{table}
\paragraph{\textbf{Box constraints on the control} $\bm{{\rm u}}$}
\noindent We now focus on the case with $-\infty \leq {\rm y} \leq \infty,\ {\rm u_a} \leq {\rm u} \leq { \rm u_b}$. Again, it is straightforward to show that such a problem admits an optimal solution (see \cite{DC}).  We run the method for different inequality bounds on the control ${\rm u}$. The results are summarized in Table \ref{Experiment Full: Constraints on control} (including all the values of the parameters used to perform the experiment).
\begin{table}[!ht]
\centering
\caption{Inequalities on the control: Varying restriction bounds (with $N = 50^3$, $\beta_x = \beta_y = 1.3$, $\alpha = 0.7$, $\gamma = 10^{-4}$, $\delta = 0.4$).\label{Experiment Full: Constraints on control}}
\scalebox{0.8}{
\begin{tabular}{rrrrrrr}
      \toprule   
    \multirow{2}{*}{$\bm{u_a}$}    & \multirow{2}{*}{$\bm{\mathcal{E}_{L^2}(y-\bar{y})}$}&  \multirow{2}{*}{\textbf{Dual Inf.}}    & {\textbf{Iterations}}      &  & \multirow{2}{*}{\textbf{Time (s)}}  \\   \cline{4-5}
 & & & {PCG} & {ADMM} & \\ \midrule
 
  $-400$ & 5.60 $\times 10^{-1}$(*) &  8.43 $\times 10^{-4}$  & 16 & 30 & 84.46 \\
    $-300$ & 5.65 $\times 10^{-1}$ &  8.79 $\times 10^{-4}$  & 19 & 22 & 72.77 \\

  $-200$ & 6.25 $\times 10^{-1}$&  4.39 $\times 10^{-4}$ & 17 & 28 & 85.55 \\
  $-100$ & 8.90 $\times 10^{-1}$&  1.49 $\times 10^{-4}$ & 18 & 65 & 205.69 \\ \bottomrule
\end{tabular}}
\end{table}
\paragraph{\textbf{Box constraints on both variables}}
\noindent Let us now consider the case where ${\rm y_a} \leq {\rm y} \leq {\rm y_b},\ {\rm u_a} \leq {\rm u} \leq {\rm u_b}$. In general, in this case one is not able to conclude that the problem admits an optimal solution. Thus, we run the method on instances for which a solution is known to exist.
\par First, we present the runs of the method for different inequality bounds in Table \ref{Experiment Full: Constraints everywhere varying bounds}. Next, we present the runs of the method for varying grid size in Table \ref{Experiment Full: Constraints everywhere grid size}. As one can observe in Table \ref{Experiment Full: Constraints everywhere grid size}, the grid size does not affect the average number of inner PCG iterations. This comes in line with our observations in Section \ref{Section A structure preserving method}. Nevertheless, it is expected that ADMM requires more iterations as the size of the problem increases. Furthermore, we can observe the first-order convergence of the numerical method, as $n$ is increased. 
\begin{table}[!ht]
\centering
\caption{Inequalities on both variables: Varying restriction bounds (with $\beta_x = \beta_y = 1.3$, $\alpha = 0.7$, $\gamma = 10^{-4}$, $\delta = 0.4$, $N = 50^3$).\label{Experiment Full: Constraints everywhere varying bounds}}
\scalebox{0.8}{
\begin{tabular}{rrrrrrrr}
      \toprule
      \multirow{2}{*}{$\bm{y_a}$} &    \multirow{2}{*}{$\bm{u_a}$}  & \multirow{2}{*}{$\bm{\mathcal{E}_{L^2}(y-\bar{y})}$}& \multirow{2}{*}{\textbf{Dual Inf.}}    & {\textbf{Iter.}}      &  & \multirow{2}{*}{\textbf{Time (s)}}  \\   \cline{5-6}
& & & & {PCG} & {ADMM} & \\ 
 \hline
  $-7$ &$-400$ & 5.60 $\times 10^{-1}$(*)& 4.88 $\times 10^{-3}$   & 10 & 36 & 74.20 \\
  $-7$  & $-200$ & 5.94 $\times 10^{-1}$&  2.35 $\times 10^{-3}$ & 11 & 38 & 80.14 \\
  $-4$ & $-350$ & 6.45 $\times 10^{-1}$&  1.99 $\times 10^{-4}$ & 18 & 126 & 412.66 \\
 $-1$  &$-400$ & 1.38 $\times 10^{0}$&  2.56 $\times 10^{-4}$ & 19 & 109 & 377.86 \\
 \bottomrule
\end{tabular}}
\end{table}

\begin{table}[!ht]
\centering
\caption{Inequalities on both variables: Varying grid size (with $\beta_x = \beta_y = 1.3$, $\alpha = 0.7$, $\gamma = 10^{-4}$, $y_a = -4$, $u_a = -350$).\label{Experiment Full: Constraints everywhere grid size}}
\scalebox{0.8}{
\begin{tabular}{rrrrrrrr}
      \toprule
       \multirow{2}{*}{$\bm{N}$} & \multirow{2}{*}{$\bm{\delta}$}  & \multirow{2}{*}{$\bm{\mathcal{E}_{L^2}(y-\bar{y})}$}& \multirow{2}{*}{\textbf{Dual Inf.}}    & {\textbf{Iter.}}      &  & \multirow{2}{*}{\textbf{Time (s)}}  \\   \cline{5-6}
& & & & {PCG} & {ADMM} & \\ 
 \hline
 $8^3$ & $2$ &3.87 $\times 10^{-1}$& 5.23 $\times 10^{-4}$   & 12 & 86 & 1.89 \\
  $16^3$ & $2$ &5.02 $\times 10^{-1}$&  8.68 $\times 10^{-5}$ & 13 & 58 & 6.06 \\
 $32^3$ & $0.4$ &6.09 $\times 10^{-1}$ &  2.94 $\times 10^{-4}$ & 16 & 62 & 75.04 \\
 $50^3$ & $0.4$ &6.45 $\times 10^{-1}$&  1.99 $\times 10^{-4}$ & 18 & 126 & 412.66 \\
  $64^3$ & $0.1$ &6.58 $\times 10^{-1}$&  3.34 $\times 10^{-3}$ & 17 & 97 & 987.12 \\
 $80^3$ & $0.1$ &6.65 $\times 10^{-1}$&  4.31 $\times 10^{-3}$ & 17 & 102 & 1,135.83 \\
 $100^3$ & $0.1$ &6.70 $\times 10^{-1}$&  4.91 $\times 10^{-3}$ & 17 & 119 & 2,436.17\\
 $128^3$ & $0.1$ &6.73 $\times 10^{-1}$&  3.49 $\times 10^{-3}$ & 17 & 169 & 9,077.08 \\

 \bottomrule
\end{tabular}}
\end{table}

\par Subsequently, we run the method for various values of the fractional derivative orders. The results are summarized in Table \ref{Experiment Full: Constraints everywhere fractional order}. As one can observe, the constraint matrix becomes ill-conditioned when $\beta$ is close to 1, due to the scaling factor in the definition of the Riesz derivative $\big(\textnormal{that is, }\frac{-1}{2 \cos(\frac{\beta \pi}{2})}\big)$. In turn, this results in an increase of the PCG iterations in the case where $\beta = 1.1$.  
\begin{table}[!ht]
\centering
\caption{Inequalities on both variables: Varying fractional derivative orders (with $N = 32^4$, $y_a = -4$, $u_a = -350$, $\gamma = 10^{-4}$).\label{Experiment Full: Constraints everywhere fractional order}}
\scalebox{0.8}{
\begin{tabular}{rrrrrrrrr}
      \toprule
      \multirow{2}{*}{$\bm{\alpha}$} &    \multirow{2}{*}{$\bm{\beta}$} & \multirow{2}{*}{$\bm{\delta}$} & \multirow{2}{*}{$\bm{\mathcal{E}_{L^2}(y-\bar{y})}$}& \multirow{2}{*}{\textbf{Dual Inf.}}    & {\textbf{Iter.}}      &  & \multirow{2}{*}{\textbf{Time (s)}}  \\   \cline{6-7}
& & & & & {PCG} & {ADMM} & \\ \midrule
  $0.1$ &$1.3$ & 0.4 & 6.46 $\times 10^{-1}$& 1.60 $\times 10^{-4}$   & 17 & 126 & 380.75 \\
  $0.3$  & $1.3$ & 0.4 & 6.46 $\times 10^{-1}$& 2.56 $\times 10^{-4}$ & 17 & 126 & 385.96\\
  $0.5$ & $1.3$ & 0.4 & 5.12 $\times 10^{-1}$ &  2.83 $\times 10^{-4}$ & 18 & 126 & 408.47 \\
  $0.9 $ & $1.3$ & 0.4 & 6.44 $\times 10^{-1}$&  3.10 $\times 10^{-4}$ & 19 & 125 & 419.46 \\
 $0.7 $ & $1.1$ & 0.4 & 6.48 $\times 10^{-1}$&  1.21 $\times 10^{-3}$ & 30 & 100 & 508.59 \\
 $0.7 $ & $1.5$ & 0.1 & 7.79 $\times 10^{-1}$&  4.23 $\times 10^{-4}$ & 15 & 96 & 275.03 \\
 $0.7 $ & $1.7$ & 0.4 & 1.04 $\times 10^{0}$&  2.46 $\times 10^{-4}$ & 13 & 113 & 275.21 \\
 $0.7 $ & $1.9$ & 0.1 & 1.36 $\times 10^{0}$&  1.36 $\times 10^{-3}$ & 8 & 108 & 180.85 \\
 \bottomrule
\end{tabular}}
\end{table}
\par Finally, we present the runs of the method for various values of the regularization parameter $\gamma$. We note at this point that as $\gamma$ is changed, the solution of the equality constrained problem is significantly altered. In light of this, we adjust the inequality constraints of the problem for each value of $\gamma$, in order to ensure that the optimal solution will lie \textit{strictly} within the bounds. That way, we are able to compare the convergence behavior of ADMM, for instances with different regularization values, $\gamma$. The results are summarized in Table \ref{Experiment Full: Constraints everywhere regularization}.

\begin{table}[!ht]
\centering
\caption{Inequalities on both variables: Varying regularization (with $N = 50^3$, $\alpha = 0.7$, $\beta = 1.3$).\label{Experiment Full: Constraints everywhere regularization}}
\scalebox{0.8}{
\begin{tabular}{rrrrrrrrrr}
      \toprule
   \multirow{2}{*}{$\bm{\gamma}$} &   \multirow{2}{*}{$\bm{y_a}$} &    \multirow{2}{*}{$\bm{u_a}$} & \multirow{2}{*}{$\bm{\delta}$} & \multirow{2}{*}{$\bm{\mathcal{E}_{L^2}(y-\bar{y})}$}& \multirow{2}{*}{\textbf{Dual Inf.}}    & {\textbf{Iter.}}      &  & \multirow{2}{*}{\textbf{Time (s)}}  \\   \cline{7-8}
& & & & & & {PCG} & {ADMM} & \\ \midrule
$10^{-2}$ & $-2$ &$-100$ & 0.1& 1.77 $\times 10^{0}$(*)& 1.38 $\times 10^{-3}$   & 11 & 87 & 190.99 \\
$10^{-4}$ &$-7$ &$-400$ & 0.4 & 5.60 $\times 10^{-1}$(*)& 4.88 $\times 10^{-3}$   & 10 & 36 & 74.20 \\
$10^{-6}$  &$-9$ & $-2,800$ & 10 & 1.28 $\times 10^{-1}$(*) & 6.03 $\times 10^{-4}$ & 8 & 47 & 71.13 \\
$10^{-8}$& $-9 $ & $-4,000$ & 100 & 1.13 $\times 10^{-1}$(*)&  2.18 $\times 10^{-4}$ & 6 & 32 & 44.70 \\
$10^{-10}$& $-9 $ & $-4,000$ & 100 & 1.13 $\times 10^{-1}$(*)&  2.18 $\times 10^{-4}$ & 5 & 32 & 40.77 \\
 \bottomrule
\end{tabular}}
\end{table}
\par We can observe that the proposed approach is sufficiently robust with respect to the problem parameters. The linear systems that have to be solved within ADMM require a small number of PCG iterations for a wide range of parameter choices. Furthermore, ADMM achieves convergence to a 4-digit accurate primal solution in a reasonable number of iterations, making the method overall efficient. In light of the generality of the approach (established in Section \ref{Section A structure preserving method}), the numerical results are very promising, and we conjecture that the proposed method can be equally effective for a very wide range of FDE optimization problems.
\section{Conclusions} \label{Section conclusions}
\noindent In this paper, we proposed the use of an Alternating Direction Method of Multipliers, for the solution of a large class of PDE-constrained convex quadratic optimization problems. Firstly, under some general assumptions, and by using the theory of Generalized Locally Toeplitz sequences, we showed that the linear system arising at every ADMM iteration preserves the GLT structure of the PDE constraints. We then associated a symbol to the aforementioned linear system while providing and analyzing some alternatives for preconditioning it efficiently. Subsequently, we focused on solving two-dimensional, time-dependent FDE-constrained optimization problems, with box constraints on the state and/or control variables. Using the Grünwald--Letnikov finite difference method, and by employing a discretize-then-optimize approach, we solved the resulting problem in the discretized variables. Given the underlying structure of such discretized problems, we designed a recursive linear algebra based on FFTs, using which we solved the associated ADMM linear systems through a Krylov subspace solver alongside a multilevel circulant preconditioner. We demonstrated how one can restrict the storage requirements to order of $N$ (where $N$ is the grid size), while requiring only $O(N \log N)$ operations for every iteration of the Krylov solver. We further verified that the number of Krylov iterations required at each ADMM iteration is independent of the grid size.  As a proof of concept, we implemented the method, and demonstrated its scalability, efficiency, and generality.
\par While the paper is focused on a special type of FDE optimization problems, we have provided a suitable methodology that has a significantly wider range of applicability. As a future research direction, we would like to employ the method, and the associated preconditioners, to various extensions of the current model, by allowing non-constant diffusion coefficients, employing higher-order discretization methods, or by solving FDEs posed in higher space-time dimensions.

\section*{Acknowledgements}
\noindent SP acknowledges financial support from a Principal's Career Development PhD scholarship at the University of Edinburgh, as well as a scholarship from A. G. Leventis Foundation. JWP acknowledges support from the Engineering and Physical Sciences Research Council (EPSRC) grant EP/S027785/1 and a Fellowship of The Alan Turing Institute. SL acknowledges financial support from a School of Mathematics PhD studentship at the University of Edinburgh, and JG acknowledges support from the EPSRC grant EP/N019652/1.
\bibliographystyle{abbrv}

\bibliography{references} 

\end{document}